\documentclass{article}

\usepackage{algorithm,algorithmic}
\usepackage{amsfonts}
\usepackage{amsmath}
\usepackage{amssymb}
\usepackage{amsthm}
\usepackage{bm}
\usepackage{booktabs}
\usepackage{enumitem}
\usepackage{forloop}
\usepackage[top=2.5cm, bottom=2.5cm, left=3cm, right=3cm]{geometry} 
\usepackage{graphicx}
\usepackage{hyperref}
\usepackage[noabbrev,nameinlink]{cleveref}
\usepackage{mathrsfs}
\usepackage{multirow}
\usepackage{ulem}

\crefformat{equation}{(#2#1#3)}
\hypersetup{colorlinks,breaklinks}
\hypersetup{
	colorlinks=true,
	linkcolor=black,
	filecolor=blue,      
	urlcolor=black,
	citecolor=black,
}

\crefname{algorithm}{Algorithm}{Algorithms}
\Crefname{algorithm}{Framework}{Frameworks}
\Crefname{exm}{Example}{Examples}
\Crefname{thm}{Theorem}{Theorems}
\Crefname{lem}{Lemma}{Lemmas}
\Crefname{prop}{Proposition}{Propositions}
\Crefname{assume}{Assumption}{Assumptions}
\Crefname{defi}{Definition}{Definitions}
\Crefname{rem}{Remark}{Remarks}
\Crefname{cond}{Condition}{Conditions}
\Crefname{ALC@unique}{Line}{Lines}
\crefname{ALC@unique}{line}{lines}
\newcounter{ct}\forloop{ct}{1}%
{\value{ct}<10}
{\crefname{enum\roman{ct}}{}{}}

\newcommand{\st}{\mathrm{s.}~\mathrm{t.}}
\newcommand{\lrbrace}[1]{\left\{#1\right\}}
\newcommand{\lrbracket}[1]{\left(#1\right)}
\newcommand{\lrsquare}[1]{\left[#1\right]}

\newcommand{\T}{\top}
\newcommand{\trace}{\mathrm{Tr}}
\newcommand{\eps}{\varepsilon}

\newcommand{\calS}{\mathcal{S}}
\newcommand{\mb}[1]{\mathbf{#1}}
\newcommand{\dom}{\mathrm{dom}}
\newcommand{\norm}[1]{\left\Vert#1\right\Vert}
\newcommand{\snorm}[1]{\Vert#1\Vert}
\newcommand{\abs}[1]{\left|#1\right|}
\newcommand{\inner}[1]{\left\langle#1\right\rangle}
\newcommand{\palm}{\textup{\textsf{PALM}}}
\newcommand{\palme}{\textup{\textsf{PALM-E}}}
\newcommand{\palmi}{\textup{\textsf{PALM-I}}}
\newcommand{\palmf}{\textup{\textsf{PALM-F}}}
\newcommand{\ssncg}{\textup{\textsf{ssncg}}}
\newcommand{\hp}{\textup{\textsf{hp}}}
\newcommand{\sadmm}{\textup{\textsf{sadmm}}}
\newcommand{\R}{\mathbb{R}}
\newcommand{\N}{\mathbb{N}}
\newcommand{\Diag}{\mathrm{Diag}}
\newcommand{\dist}{\mathrm{dist}}
\newcommand{\bbar}{\overline}
\newcommand{\D}{\mathrm{D}}
\newcommand{\blam}{\bm{\lambda}}
\newcommand{\h}{\mathbf{h}}
\newcommand{\uu}{\mathbf{u}}
\newcommand{\w}{\mathbf{w}}
\newcommand{\x}{\mathbf{x}}
\newcommand{\y}{\mathbf{y}}
\newcommand{\z}{\mathbf{z}}
\newcommand{\PP}{\mathscr{P}}
\newcommand{\one}{\mathbf{1}}
\newcommand{\tabincell}[2]{\begin{tabular}{@{}#1@{}}#2\end{tabular}}

\newtheorem{assume}{Assumption}[section]

\newtheorem{cond}{Condition}[section]
\newtheorem{coro}{Corollary}[section]
\newtheorem{defi}{Definition}[section]
\newtheorem{lem}{Lemma}[section]
\newtheorem{prop}{Proposition}[section]
\newtheorem{thm}{Theorem}[section]
\newtheorem{exm}{Example}[section]
\newtheorem{rem}{Remark}[section]

\allowdisplaybreaks[4]

\DeclareMathOperator*{\argmin}{arg\,min}

\numberwithin{equation}{section}

\title{The Convergence Properties of Infeasible Inexact Proximal Alternating Linearized Minimization}
\author{
	Yukuan Hu\thanks{State Key Laboratory of Scientific and Engineering Computing, Academy of Mathematics and Systems Science, Chinese Academy of Sciences, and University of Chinese Academy of Sciences, China (\href{mailto:ykhu@lsec.cc.ac.cn}{ykhu@lsec.cc.ac.cn}, \href{mailto:liuxin@lsec.cc.ac.cn}{liuxin@lsec.cc.ac.cn}). The research was supported in part by the National Natural Science Foundation of China (No. 12125108, 11971466, 11991021, 11991020, 12021001, and 11688101), Key Research Program of Frontier Sciences, Chinese Academy of Sciences (No. ZDBS-LY-7022), and the CAS AMSS-PolyU Joint Laboratory in Applied Mathematics.}
	\and Xin Liu\footnotemark[1]
}
\date{\today}

\begin{document}
	
	\maketitle
	
	\begin{abstract}
		The proximal alternating linearized minimization method (\palm) suits well for solving block-structured optimization problems, which are ubiquitous in real applications. In the cases where subproblems do not have closed-form solutions, e.g., due to complex constraints, infeasible subsolvers are indispensable, giving rise to an infeasible inexact \palm~(\palmi). Numerous efforts have been devoted to analyzing feasible \palm, while little attention has been paid to \palmi. The usage of \palmi~thus lacks theoretical guarantee. The essential difficulty of analyses consists in the objective value nonmonotonicity induced by the infeasibility. We study in the present work the convergence properties of \palmi. In particular, we construct a surrogate sequence to surmount the nonmonotonicity issue and devise an implementable inexact criterion. Based upon these, we manage to establish the stationarity of any accumulation point and, moreover, show the iterate convergence and the asymptotic convergence rates under the assumption of the \L ojasiewicz property. The prominent advantages of \palmi~on CPU time are illustrated via numerical experiments on problems arising from quantum physics and 3D anisotropic frictional contact. \\[-0.5em]
		
		
	\end{abstract}

	\section{Introduction}\label{sec:introduction}
	
	\par In this work, we focus on the minimization problem with block structure as
	\begin{equation}
		\min_{\z\in\otimes_{i=1}^n\R^{m_i}}~f(\x_1,\ldots,\x_n),\quad\st~~\x_i\in\calS_i:=\{\w_i\in\R^{m_i}:\h_i(\w_i)\le0\},~i=1,\ldots,n,
		\label{eqn:original prob}
	\end{equation}
	where $f:\otimes_{i=1}^n\R^{m_i}\to\R$ is differentiable and not necessarily convex, $\z:=(\x_1,\ldots,\x_n)$; for $i=1,\ldots,n$, $\x_i\in\R^{m_i}$, $\h_i:=(h_{i,1},\ldots,h_{i,p_i})^\T:\R^{m_i}\to\R^{p_i}$ is convex differentiable, and $m_i$, $p_i\in\N$. Problems sharing this form are ubiquitous; see, e.g., \cite{bonettini2018block,he2010approximation,hu2021global,kuvcera2008convergence,liu2020topology} and the references within. We also adopt an extended-valued form of \cref{eqn:original prob}
	\begin{equation}
		\min_{\z\in\otimes_{i=1}^n\R^{m_i}}~F(\x_1,\ldots,\x_n):=f(\x_1,\ldots,\x_n)+\sum_{i=1}^n\delta_{\calS_i}(\x_i),
		\label{eqn:extended form}
	\end{equation}
	where $\delta_{\calS_i}$ stands for the indicator function of $\calS_i$, i.e., $\delta_{\calS_i}(\mb{w})$ equals $0$ if $\mb{w}\in\calS_i$ otherwise $\infty$.
	
	\par In view of the block structure of \cref{eqn:original prob}, we consider the Proximal Alternating Linearized Minimization method (\palm); see \Cref{frame:ipalm}, where we impose flexible conditions on the iterate sequence.
	
	\begin{algorithm}[htbp]
		\floatname{algorithm}{Framework}
		\centering
		\caption{\palm~for solving \cref{eqn:original prob}}
		\label{frame:ipalm}
		\begin{algorithmic}[1]
			\REQUIRE{Initial point $\z^{(0)}=(\x_i^{(0)})_{i=1}^n\in\otimes_{i=1}^n\R^{m_i}$, proximal parameters $\{\sigma_i^{(0)}>0\}_{i=1}^n$.}
			\ENSURE{An approximate solution $\z^{(k)}:=(\x_i^{(k)})_{i=1}^n\in\otimes_{i=1}^n\R^{m_i}$.}
			\STATE{Set $k:=0$.}
			\WHILE{\textit{certain conditions not satisfied}}
			\FOR{$i=1,\ldots,n$}
			\STATE{Solve the $i$-th proximal linearized subproblem\vspace{-3mm}
				\begin{equation}
					\min_{\x_i\in\calS_i}~\inner{\nabla_if(\x_{< i}^{(k+1)},\x_{\ge i}^{(k)}),\x_i-\x_i^{(k)}}+\frac{\sigma_i^{(k)}}{2}\snorm{\x_i-\x_i^{(k)}}^2
					\label{eqn:prox-linear subprob}
					\vspace{-2mm}
				\end{equation}
				to obtain $\x_i^{(k+1)}\in\R^{m_i}$ fulfilling \textit{certain conditions}. 
				\STATE{Update the $i$-th proximal parameter $\sigma_i^{(k)}$ to $\sigma_i^{(k+1)}>0$ if necessary.}
			}
			\ENDFOR
			\STATE{Set $k:=k+1$.}
			\ENDWHILE
		\end{algorithmic}
	\end{algorithm}
	
	\par When the subproblem \cref{eqn:prox-linear subprob} is exactly solved, we obtain the Exact \palm~(\palme). With properly chosen proximal parameters, one could derive sufficient reduction over the objective value sequence. Based upon this point, the stationarity of any accumulation point follows. This methodology applies to more general frameworks, such as the block successive minimization in \cite{razaviyayn2013unified} and the Bregman-distance-based block coordinate proximal gradient methods in \cite{hua2016block,wang2018block}. Furthermore, with the aid of the \L ojasiewicz property that is shared by a broad swath of functions, one could obtain the iterate convergence in more generic settings; see, e.g., \cite{bolte2014proximal2,Xu2013A2}.
	
	\par It is not difficult to check that solving \cref{eqn:prox-linear subprob} in \Cref{frame:ipalm} amounts to projecting the point
	\begin{equation*}
		\tilde\x_i^{(k)}:=\x_i^{(k)}-\frac{1}{\sigma_i^{(k)}}\nabla_if(\x_{<i}^{(k+1)},\x_{\ge i}^{(k)})
	\end{equation*}
	onto $\calS_i$. More often, however, the projection is not of closed-form expression. In these contexts, \textit{inexactly} solving \cref{eqn:prox-linear subprob} becomes a much more pragmatic option. Efficient subsolvers for \cref{eqn:prox-linear subprob} could hence be brought to bear. 
	
	\par When the subsolvers inexactly solve \cref{eqn:prox-linear subprob} and yield $\x_i^{(k)}\in\calS_i$ throughout iterations, we obtain the Feasible inexact \palm~(\palmf). Most works in this setting \textit{enforce the monotonicity} of the objective value sequence. Some of them (repeatedly), in one outer iteration, solve the subproblem inexactly to obtain a descent direction and then perform line search; see, e.g., \cite{bonettini2018block,yang2019inexact}. In \cite{hua2016block}, the authors treat the solution error as an additional term in the kernel function defining the Bregman distance, and then impose assumptions on the solution errors to invoke the results established in the exact settings. In \cite{frankel2015splitting,ochs2019unifying}, the authors put flexibility in solving \cref{eqn:prox-linear subprob} in the sense that the relative error conditions are relaxed while maintaining the sufficient reduction property. 
	
	\par In contrast, little attention has been paid to the Infeasible inexact \palm~(\palmi), where the subsolvers inexactly solve \cref{eqn:prox-linear subprob} but not necessarily give $\x_i^{(k)}\in\calS_i$. However, when the constraints describing $\{\calS_i\}_{i=1}^n$ are complicated, infeasible subsolvers, such as (primal-)dual or penalty methods, are indispensable. To illustrate, we list two instances below, along with some state-of-art algorithms for computing the projections. 
	
	\begin{exm}[Linear constraints]\label{exm:linear constraint}
		The feasible region $\calS_i$ is the Birkhoff polytope $\calS_i:=\{W\in\R^{m_i\times m_i}:W\one=\one,W^\T\one=\one,W\ge0\}$, where $\one$ stands for the all-one vector in $\R^{m_i}$. This type of feasible region shows up frequently in applications such as optimal transport problems \cite{peyre2019computational} and electronic structure calculation \cite{hu2021global}. Since the number of constraints describing $\calS_i$ is much less than the underlying space dimension (given even moderate $m_i$), it is more reasonable to solve the subproblem \cref{eqn:prox-linear subprob} from the dual perspective. To this end, we could invoke the semismooth Newton method proposed in \cite{li2020efficient}. By exploiting the structure of $\calS_i$, high efficiency can be achieved \cite{hu2021global}. Nevertheless, the recovered primal solution is infeasible. 
	\end{exm}
	
	\begin{exm}[Nonlinear constraints]\label{exm:nonlinear constraint}
		The feasible region $\calS_i$ is an ellipsoid in $\R^{m_i}$, namely, $\calS_i:=\{\w\in\R^{m_i}:\frac{1}{2}\w^\T A_i\w+\mb{b}_i^\T\w\le\alpha_i\}$, where $I\ne A_i\in\R^{m_i\times m_i}$ is positive definite symmetric, $\mb{b}_i\in\R^{m_i}$, and $\alpha_i>0$. Projecting a point onto an ellipsoid emerges as one of the fundamental problems in convex analysis and numerical algorithms with applications in topology optimization \cite{liu2020topology} and 3D contact problems with an anisotropic friction \cite{kuvcera2008convergence} as well as relations to polynomial optimization \cite{he2010approximation}, just to mention a few. When $\mb{b}_i=0$, it is also related to the trust region subproblem in nonlinear optimization \cite{powell1991trust}. We refer interested readers to a recent work \cite{jia2017comparison}, where an alternating direction method of multipliers is proposed to solve the reformulated problem. The primal variables are then not necessarily feasible upon termination. The proposed method is reported to outperform the existent feasible one in \cite{dai2006fast}.
	\end{exm}
	
	\par Owing to the infeasibility, the objective value sequence is not ensured to be monotonic, while the sufficient reduction of the objective value is presumably crucial in proving the stationarity of any accumulation point. The only work exploring the convergence properties of \palmi~goes to \cite{frankel2015splitting}. The obtained results, however, might be of only theoretical values. The authors impose the following hypothesis: there exist $\beta_1$, $\beta_2>0$ such that, for $i=1,\ldots,n$ and $k\ge0$,
	\begin{equation}
		\left\{\begin{array}{l}
			\sum_{j=1}^{i-1}\snorm{\x_j^{(k+1)}-\bar\x_j^{(k+1)}}+\sum_{j=i}^n\snorm{\x_j^{(k)}-\bar\x_j^{(k)}}\le\beta_1\snorm{\bar\x_i^{(k+1)}-\bar\x_i^{(k)}},\\
			\inner{\x_i^{(k)}-\bar\x_i^{(k)},\bar\x_i^{(k+1)}-\bar\x_i^{(k)}}\le\beta_2\snorm{\bar\x_i^{(k+1)}-\bar\x_i^{(k)}}^2,
		\end{array}\right.
		\label{eqn:existing inexact criterion}
	\end{equation}
	where $\bar\x_i^{(k+1)}$ is the unique solution of \cref{eqn:prox-linear subprob}, defined as
	\begin{equation}
		\bar{\x}_i^{(k+1)}:=\argmin_{\x_i\in\calS_i}~\inner{\nabla_if(\x_{<i}^{(k+1)},\x_{\ge i}^{(k)}),\x_i-\x_i^{(k)}}+\frac{\sigma_i^{(k)}}{2}\snorm{\x_i-\x_i^{(k)}}^2.
		\label{eqn:block optimal}
	\end{equation}
	Based upon \cref{eqn:existing inexact criterion}, they establish a sufficient reduction result over the objective value sequence $\{f(\bar\z^{(k)})\}$, where $\bar\z^{(k)}:=(\bar{\x}_1^{(k)},\ldots,\bar\x_n^{(k)})$. It is unclear how to fulfill \cref{eqn:existing inexact criterion} in practice for the reasons that (i) $\bar\x_i^{(k)}$ and $\bar\x_i^{(k+1)}$ cannot be computed, not to mention $\snorm{\bar\x_i^{(k+1)}-\bar\x_i^{(k)}}$; (ii) $\snorm{\bar\x_i^{(k+1)}-\bar\x_i^{(k)}}$ is needed for obtaining $\{\x_j^{(k)}\}_{j=i}^n$. Unfortunately, the authors in \cite{frankel2015splitting} do not discuss these points. In consequence, the convergence properties of \palmi~remain to be investigated, particularly with \textit{implementable inexact criteria}. This is essential in providing a theoretical guarantee for the usage of efficient infeasible subsolvers.
	
	\subsection{Contributions}
	
	\par In this work, we establish the convergence properties of \palmi~for solving \cref{eqn:original prob}. In particular, we 
	\begin{enumerate}[label=(\roman*),topsep=0pt, parsep=0pt, itemsep=0pt]
		\item control the solution errors when solving \cref{eqn:prox-linear subprob} with a prescribed nonnegative sequence $\{\eps^{(k)}\}$ and an error bound that is computable for any subsolvers. Our inexact criterion is thus much more pragmatic than that in \cite{frankel2015splitting};
		
		\item construct a nonincreasing surrogate sequence $\{v^{(k)}\}$ to surmount the objective value nonmonotonicity issue. The objective value sequence is allowed to fluctuate, favoring more extensive flexibility than most existing works;
		
		\item establish the convergence properties, including the iterate convergence to stationarity and the asymptotic iterate convergence rates, of \palmi~with the help of the \L ojasiewicz property of $F$ in \cref{eqn:extended form}. These results are new to the best of our knowledge;
		
		\item illustrate the considerable advantages of \palmi~on CPU time over \palme~and \palmf~through numerical experiments on problems arising from quantum physics and 3D anisotropic frictional contact.
	\end{enumerate}
	
	\par Before concluding this subsection, we gather some of the established asymptotic convergence rates in \Cref{tab:rates} to showcase the comparison with existing works, where $\theta$ is the \L ojasiewicz exponent of $F$ associated with a compact set.
	\begin{table}[htbp]
		\centering
		\caption{Asymptotic convergence rates of \palm~under different settings.}
		\label{tab:rates}
		\begin{tabular}{cclll}
			\toprule
			$\theta$ & $\eps^{(k)}$ & Extra assumptions & Rates & References \\\midrule
			\multirow{3}{*}[-1.5ex]{$0$} & $0$ & - & Finite termination & \cite{bolte2014proximal2,Xu2013A2} \\\cmidrule[0.01mm]{2-5}
			& $\tilde\rho^k$ & $\tilde\rho\in(0,1)$ & $\mathcal{O}(\rho_1^k)$, where $\rho_1\in(0,1)$ & This paper (\Cref{thm:convergence rates exponential})\\
			& $\frac{1}{(k+1)^\ell}$ & $\ell\in(1,\infty)$ & $\mathcal{O}\lrbracket{k^{-(\ell-1)}}$ & This paper (\Cref{thm:convergence rates sublinear}) \\\midrule[0.1mm]
			\multirow{3}{*}[-1ex]{$(0,\frac{1}{2}]$} & $0$ & - & $\mathcal{O}(\rho_2^k)$, where $\rho_2\in(0,1)$ & \cite{bolte2014proximal2,Xu2013A2} \\\cmidrule[0.01mm]{2-5}
			& $\tilde\rho^k$ & $\tilde\rho\in(0,1)$ & $\mathcal{O}(\rho_3^k)$, where $\rho_3\in(0,1)$ & This paper (\Cref{thm:convergence rates exponential})\\
			& $\frac{1}{(k+1)^\ell}$ & $\ell\in(1,\infty)$ & $\mathcal{O}\lrbracket{k^{-(\ell-1)}}$ & This paper (\Cref{thm:convergence rates sublinear})\\\midrule[0.1mm]
			\multirow{3}{*}[-4.5ex]{$(\frac{1}{2},1)$} & $0$ & - & $\mathcal{O}\lrbracket{k^{-\frac{1-\theta}{2\theta-1}}}$ & \cite{bolte2014proximal2,Xu2013A2}\\\cmidrule[0.01mm]{2-5}
			& $\tilde\rho^k$ & $\tilde\rho\in(0,1)$ & $\mathcal{O}\lrbracket{k^{-\frac{1-\theta}{2\theta-1}}}$ & This paper (\Cref{thm:convergence rates exponential})\\
			& $\frac{1}{(k+1)^\ell}$ & $\ell\in(1,\infty)$ & \tabincell{l}{$\mathcal{O}\lrbracket{k^{-\frac{1-\theta}{2\theta-1}}}$ if $\ell\ge\frac{\theta}{2\theta-1}$\\$\mathcal{O}\lrbracket{k^{-(\ell-1)}}$ if $\ell<\frac{\theta}{2\theta-1}$} & This paper (\Cref{thm:convergence rates sublinear})\\\bottomrule
		\end{tabular}
	\end{table}
	
	\subsection{Notations and Organization}\label{subsec:notation}
	
	\par This paper presents scalars, vectors, and matrices by lower-case letters, bold lower-case letters, and upper-case letters, respectively. The notation $\one$ stands for the all-one vector with proper dimension. The notations $\inner{\cdot,\cdot}$ and $\snorm{\cdot}$ calculate, respectively, the standard inner product and the norm of vectors in the ambient Euclidean space. We use $\Diag(\cdot)$ to form a diagonal matrix with the input vector. 
	
	\par We use subscripts to denote the components or blocks of vectors or matrices; e.g., $\x_i$ is the $i$-th variable block. Occasionally for brevity, we make abbreviation for the aggregation of variable blocks; e.g., $\x_{<i}:=(\x_1,\ldots,\x_{i-1})$ and $\x_{>i}:=(\x_{i+1},\ldots,\x_n)$ (clearly, $\x_{<0}$, $\x_{>n}$ are null variable blocks, which may be used for notational ease). Likewise, we can define $\x_{\le i}$, $\x_{\ge i}$, $\x_{(j,i)}$, $\x_{(j,i]}$, $\x_{[j,i)}$, and $\x_{[j,i]}$ (the latter four are also null if the index sets in the subscript are empty). 
	
	\par For a function $h$, $\nabla h$ (resp. $\partial h$) is the gradient (resp. subdifferential) of $h$ at certain point where $h$ is differentiable (resp. subdifferentiable). We add a subscript to indicate the block to which the derivative is taken with respect; e.g., $\nabla_i$. For a differentiable mapping $\h:\R^m\to\R^p$, we denote by $\nabla\h:\R^m\to\R^{m\times p}$ its Jacobian. The notation $\delta_{\calS}$ stands for the indicator function of a set $\calS$, i.e., $\delta_{\calS}(\mb{w})$ equals $0$ if $\mb{w}\in\calS$ otherwise $\infty$. We denote the effective domain of a function $h$ by $\dom(h):=\{\y:h(\y)<\infty\}$. With a slight abuse of notation, the domain of its subdifferential is $\dom(\partial h):=\{\y:\partial h(\y)\ne\emptyset\}$.
	
	\par Given a set $\calS$ and a point $\mb{w}$, $\dist(\mb{w},\calS):=\inf_{\mb{w}'\in\calS}\snorm{\mb{w}-\mb{w}'}$ stands for the distance from $\mb{w}$ to $\calS$. If the set $\calS$ is nonempty closed, we define the projection operator $\PP_{\calS}$ onto $\calS$ as $\PP_{\calS}(\mb{w})\in\argmin_{\mb{w}'\in\calS}\snorm{\mb{w}-\mb{w}'}$. The notation ``$\otimes$'' denotes the Cartesian product of sets or spaces. The notation $B_{\eta}(\x)$ with $\eta>0$ refers to the closed ball in the ambient space centered at $\x$ with radius $\eta$.
	
	\par We organize this paper as follows: in \cref{sec:preliminaries}, we present some definitions used throughout this work and introduce the \L ojasiewicz property. The complete description of \palmi~is described in \cref{sec:palmi}, including details on the inexact criterion in use. We establish the global convergence properties of \palmi~in \cref{sec:convergence of ipalm}, including a weak and a strong form. We analyze the asymptotic convergence rates of \palmi~under different settings in \cref{sec:convergence rates of ipalm}. Numerical experiments are detailed in \cref{sec:numerical experiments}. Some concluding remarks are drawn in \cref{sec:conclusions}.
	
	\section{Preliminaries}\label{sec:preliminaries}
	
	\par We collect several notions from convex analysis as well as the \L ojasiewicz property in this section. 
	
	\begin{defi}[\cite{rockafellar2009variational2}]\label{def:subdifferential}
		Let $G:\mathbb{E}\to(-\infty,\infty]$ be a proper closed function, where $\mathbb{E}$ is an Euclidean space. For a given $\x\in\dom(G)$, the Fr\'echet subdifferential of $G$ at $\x$, denoted by $\partial G(\x)$, is defined as
		$$\partial G(\x):=\lrbrace{\uu\in\mathbb{E}:\liminf_{\y\ne\x,\y\to\x}\frac{G(\y)-G(\x)-\inner{\uu,\y-\x}}{\snorm{\y-\x}}\ge0}.$$
		When $\x\notin\dom(G)$, we simply set $\partial G(\x)=\emptyset$. When $\partial G(\x)$ is a singleton, we say that $G$ is Fr\'echet differentiable at $\x$ and denote the derivative by $\nabla G(\x)$. 
	\end{defi}
	
	\begin{rem}\label{rem:subdiff}
		\begin{enumerate}[label=(\roman*),topsep=0pt, parsep=0pt, itemsep=0pt]
			\item If $G:\mathbb{E}\to(-\infty,\infty]$ is proper closed convex, then, 
			$$\partial G(\x)=\lrbrace{\uu\in\mathbb{E}:G(\y)-G(\x)\ge\inner{\uu,\y-\x},~\forall~\y\in\mathbb{E}},\quad\forall~\x\in\dom(G).$$
			
			\item If $G:\mathbb{E}\to(-\infty,\infty]$ and $H:\mathbb{E}\to(-\infty,\infty]$ are proper closed functions, and $G$ is Fr\'echet differentiable at $\x$, then $\partial(G+H)(\x)=\nabla G(\x)+\partial H(\x)$. 
			
			\item If $G:\mathbb{E}\to(-\infty,\infty]$ is proper closed and $0\in\partial G(\x)$, we call $\x$ a stationary point of $G$. 
		\end{enumerate}
	\end{rem}
	
	\par With the definition of subdifferential in place, we recall the \L ojasiewicz property given in \cite{attouch2009convergence}. The \L ojasiewicz property is introduced first in \cite{lojasiewicz1993sur} on the real analytic functions, and then is extended to the functions on the $o$-minimal structure in \cite{kurdyka1998gradients} and to the nonsmooth subanalytic functions in \cite{bolte2007lojasiewicz} under the name of Kurdyka-\L ojasiewicz property afterwards \cite{attouch2010proximal,bolte2014proximal2,Xu2013A2}.
	
	\begin{defi}[\cite{attouch2009convergence}]\label{def:KL property}
		Let $G:\mathbb{E}\to(-\infty,\infty]$ be a proper closed function, where $\mathbb{E}$ is an Euclidean space. The function $G$ is said to have the \L ojasiewicz property at some stationary point $\bar\x$ if there exist $c>0$, $\theta\in[0,1)$, and $\eta>0$ such that, for any $\x\in B_{\eta}(\bar\x)$,
		$$\abs{G(\x)-G(\bar\x)}^\theta\le c\cdot\dist(0,\partial G(\x)),$$
		where we adopt the convention $0^0=0$ if $\theta=0$, and therefore, if $\abs{G(\x)-G(\bar\x)}^0=0$, we have $G(\x)=G(\bar\x)$. We call $\theta$ the \L ojasiewicz exponent of $G$ at $\bar\x$.
	\end{defi}
	
	\begin{rem}
		Existing works have revealed some valid examples. For instance, the real-analytic functions \cite{lojasiewicz1963propriete}, the convex functions fulfilling certain growth conditions \cite{bolte2007lojasiewicz}, and the semialgebraic functions \cite{attouch2010proximal}. We refer readers to \cite{attouch2010proximal} for a comprehensive collection. Notably, the class of semialgebraic functions covers a wide range of functions commonly used by the optimization community.
	\end{rem}
	
	\par In \cite{attouch2009convergence}, the authors provide the following uniformized version of the \L ojasiewicz property, which could be shown using the Heine-Borel theorem.
	
	\begin{lem}[\cite{attouch2009convergence}]\label{lem:KL}
		Let $G: \mathbb{E}\to(-\infty,\infty]$ be a proper closed function, where $\mathbb{E}$ is an Euclidean space. Let $\Omega\subseteq\mathbb{E}$ be a connected compact set consisting of the stationary points of $G$. Assume that $G$ has the \L ojasiewicz property at each stationary point. Then $G$ is constant on $\Omega$ and there exist $c$, $\eta>0$, and $\theta\in[0,1)$  such that, for any $\bar\x\in\Omega$ and $\x\in\{\y\in\mathbb{E}:\dist(\y,\Omega)\le\eta\}$,
		$$\abs{G(\x)-G(\bar\x)}^\theta\le c\cdot\dist(0,\partial G(\x)).$$
		We call $\theta$ the \L ojasiewicz exponent of $G$ (associated with $\Omega$). 
	\end{lem}
	
	When $\x$ satisfies both $\dist(\x,\Omega)<\eta$ and $\abs{G(\x)-G(\bar\x)}<1$, one could lift the \L ojasiewicz exponent to a larger value, as observed in \cite{chill2003lojasiewicz,li2021convergence}. 
	
	\begin{coro}[\cite{chill2003lojasiewicz,li2021convergence}]\label{coro:larger theta}
		Let $G:\mathbb{E}\to(-\infty,\infty]$ be a proper closed function, where $\mathbb{E}$ is an Euclidean space. Let $\Omega\subseteq\mathbb{E}$ be a connected compact set consisting of the stationary points of $F$. Assume that $G$ has the \L ojasiewicz property at each stationary point. Let $c$, $\eta>0$ and $\theta\in[0,1)$ be the constants associated with $G$ and $\Omega$ in \Cref{lem:KL}. Then, for any $\bar\theta\in[\theta,1)$, for all $\bar\x\in\Omega$ and all $\x\in\{\y\in\mathbb{E}:\dist(\y,\Omega)<\eta\}\cap\{\y:\abs{G(\y)-G(\bar\x)}<1\}$, 
		$$\abs{G(\x)-G(\bar\x)}^{\bar\theta}\le c\cdot\dist(0,\partial G(\x)).$$
		We call $\bar\theta$ the lifted \L ojasiewicz exponent of $G$ (associated with $\Omega)$. 
	\end{coro}
	
	\par In the sequel, we distinguish the lifted exponents from the unlifted ones using \textit{overlines} as above. We end this section with a list of inequalities for reference, whose proof is omitted. 
	
	\begin{lem}\label{lem:inequality lemma}
		\begin{enumerate}[label=(\roman*),topsep=0pt, parsep=0pt, itemsep=0pt]
			\item For any $a_i\ge0$, $i=1,\ldots,n$, 
			$$\sqrt[n]{\prod_{i=1}^na_i}\le\frac{1}{n}\sum_{i=1}^na_i\le\sqrt{\frac{1}{n}\sum_{i=1}^na_i^2}.$$
			
			\item For any $a$, $b\ge0$, and $p\in(1,\infty)$, $(a+b)^p\le2^{p-1}(a^p+b^p)$.
			
			\item For any $a$, $b\ge0$, and $p\in(0,1)$, $(a+b)^p\le a^p+b^p$.
		\end{enumerate}
	\end{lem}

	\section{\palmi}\label{sec:palmi}
	
	\par We give the complete description of \palmi~in this section; see \cref{alg:palmi}. 
	\begin{algorithm}[htbp]
		\centering
		\caption{\palmi~for solving \cref{eqn:original prob}}
		\label{alg:palmi}
		\begin{algorithmic}[1]
			\REQUIRE{Initial point $\z^{(0)}=(\x_i^{(0)})_{i=1}^n\in\otimes_{i=1}^n\R^{m_i}$, $\bar\eps>0$, nonnegative sequence $\{\eps^{(k)}\le\bar\eps\}$, $M_u\ge M_l>0$, initial proximal parameters $\{\sigma_i^{(0)}\in[M_l,M_u]\}_{i=1}^n$.}
			\ENSURE{An approximate solution $\z^{(k)}:=(\x_i^{(k)})_{i=1}^n\in\otimes_{i=1}^n\R^{m_i}$.}
			\STATE{Set $k:=0$.}
			\WHILE{\textit{certain conditions not satisfied}}
			\FOR{$i=1,\ldots,n$}
			\STATE{Solve the $i$-th proximal linearized subproblem\vspace{-3mm}
				\begin{equation}
					\min_{\x_i\in\calS_i}~\inner{\nabla_if(\x_{< i}^{(k+1)},\x_{\ge i}^{(k)}),\x_i-\x_i^{(k)}}+\frac{\sigma_i^{(k)}}{2}\snorm{\x_i-\x_i^{(k)}}^2
					\label{eqn:prox-linear subprob new}
					\vspace{-2mm}
				\end{equation}
				to obtain $\x_i^{(k+1)}\in\R^{m_i}$ such that there exists $\blam_i^{(k+1)}\in\R_+^{p_i}$ fulfilling \vspace{-3mm}
				$$\sqrt{r_i(\x_i^{(k+1)},\blam_i^{(k+1)},\tilde\x_i^{(k)})}\le\eps^{(k)}.$$\vspace{-6mm}}
			\STATE{Update the $i$-th proximal parameter $\sigma_i^{(k)}$ to $\sigma_i^{(k+1)}\in[M_l,M_u]$ if necessary.}
			\ENDFOR
			\STATE{Set $k:=k+1$.}
			\ENDWHILE
		\end{algorithmic}
	\end{algorithm}
	Compared with \Cref{frame:ipalm}, we specify the inexact criterion for subsolvers as well as some additional parameters for determining $\{\sigma_i^{(k)}\}$. The constants $M_l$ and $M_u$, defined later in \cref{sec:convergence of ipalm}, are associated with $f$, $\{\calS_i\}_{i=1}^n$, and $\{\eps^{(k)}\}$. For $i=1,\ldots,n$, the residual function $r_i:\R^{m_i}\times\R^{p_i}\times\R^{m_i}\to\R_+$ is defined as
	\begin{align}
		r_i(\x_i,\blam_i,\tilde\x_i):=&\max\lrbrace{\inner{\x_i,\x_i-\tilde\x_i+\nabla\h_i(\x_i)\blam_i},0}+\norm{\x_i-\tilde\x_i+\nabla\h_i(\x_i)\blam_i}_\infty\nonumber\\
		&+\norm{\max\lrbrace{\h_i(\x_i),0}}_\infty+\max\lrbrace{-\inner{\blam_i,\h_i(\x_i)},0}.\label{eqn:residual function}
	\end{align}
	
	\begin{rem}
		If we employ primal-dual subsolvers to solve \cref{eqn:prox-linear subprob new}, the dual variables can just be taken as $\blam_i^{(k+1)}$ in \palmi. Otherwise, one could solve the following linear programming
		$$\min_{\blam_i\in\R_+^{p_i}}~0,\quad\st\left\{\begin{array}{l} 			
			\inner{\x_i^{(k+1)},\x_i^{(k+1)}-\tilde\x_i^{(k)}+\nabla\h_i(\x_i^{(k+1)})\blam_i}\le\frac{\eps^{(k)}}{4},\\
			-\frac{\eps^{(k)}}{4}\one\le\x_i^{(k+1)}-\tilde\x_i^{(k)}+\nabla\h_i(\x_i^{(k+1)})\blam\le\frac{\eps^{(k)}}{4}\one,\\
			-\inner{\blam,\h_i(\x_i^{(k+1)})}\le\frac{\eps^{(k)}}{4}
		\end{array}\right.$$
		at any subiteration where $\snorm{\max\{\h_i(\x_i^{(k+1)}),0\}}_\infty\le\frac{\eps^{(k)}}{4}$. 
	\end{rem}
	
	\par The inexact criterion adopted in \palmi~guarantees an error bound for \cref{eqn:prox-linear subprob new} under certain conditions.
	
	\begin{lem}\label{lem:error bound}
		Suppose that $f$ is continuously differentiable with respect to each variable block over $\otimes_{i=1}^n\bar\calS_i$, where, for $i=1,\ldots,n$, $\bar\calS_i:=\{\w_i\in\R^{m_i}:\dist(\w_i,\calS_i)\le\bar\eps\}$. For $i=1,\ldots,n$, assume that $\calS_i$ is convex compact and $\h_i$ is a linear mapping or satisfies the Slater constraint qualification, i.e., $\h_i(\hat\x_i)<0$ for some $\hat\x_i\in\R^{m_i}$. Assume further that, for $i=1,\ldots,n$, the Hoffman-like bound 
		\begin{equation}
			\dist(\x_i,\calS_i)\le\tilde c_i\snorm{\max\{\h_i(\x_i),0\}},\quad\forall~\x_i\in\tilde\calS_i:=\lrbrace{\w_i\in\R^{m_i}:\dist(\w_i,\bar\calS_i)\le\frac{\bar{\mathcal{M}_i}}{M_l}}
			\label{eqn:hoffman-like error bound}
		\end{equation}
		holds for some constant $\tilde c_i\ge0$, where $\bar{\mathcal{M}}_i:=\sup_{\z\in\otimes_{i=1}^n\bar\calS_i}\norm{\nabla_if(\z)}$. Let $\{\z^{(k)}\}$ be the iterate sequence generated by \palmi. Then there exists a constant $\omega\ge0$ such that $\snorm{\z^{(k+1)}-\bar\z^{(k+1)}}\le\omega\eps^{(k)}$ holds for any $k\ge0$.
	\end{lem}
	
	\begin{proof}
		By \cite[Theorem 2.2]{mangasarian1988error} and the assumptions on $\{\h_i\}_{i=1}^n$, it holds, for $i=1,\ldots,n$ and $k\ge0$, that
		\begin{multline*}
			\snorm{\x_i^{(k+1)}-\bar\x_i^{(k+1)}}^2\le\max\lrbrace{-\inner{\blam_i^{(k+1)},\h(\x_i^{(k+1)})},0}+\omega_{i,1}\norm{\x_i^{(k+1)}-\tilde\x_i^{(k)}+\nabla\h_i(\x_i^{(k+1)})\blam_i^{(k+1)}}_\infty\\
			+\max\lrbrace{\inner{\x_i^{(k+1)},\x_i^{(k+1)}-\tilde\x_i^{(k)}+\nabla\h_i(\x_i^{(k+1)})\blam_i^{(k+1)}},0}+\omega_{i,2}^{(k)}\norm{\max\{\h_i(\x_i^{(k+1)}),0\}}_\infty,
		\end{multline*}
		where $\omega_{i,1}:=\max_{\x_i\in\calS_i}\snorm{\x_i}_1$, $\omega_{i,2}^{(k)}:=\min_{\blam_i\in\mathcal{W}_i^{(k)}}\snorm{\blam_i}_1$, and $\mathcal{W}_i^{(k)}\subseteq\R_+^{p_i}$ is the set of optimal Lagrange multipliers of \cref{eqn:prox-linear subprob new}. By the Hoffman-like bound \cref{eqn:hoffman-like error bound}, we obtain from \cite[Proposition 3]{bertsekas1999error} that $\sup_k\omega_{i,2}^{(k)}<\infty$. Let $\omega_i:=\max\lrbrace{1,\omega_{i,1},\sup_k\omega_{i,2}^{(k)}}$. The proof is then complete after letting $\omega:=\max_i\sqrt{\omega_i}$ and noticing \cref{eqn:residual function}.
	\end{proof}

	\begin{rem}
		\par Compared with the hypothesis \cref{eqn:existing inexact criterion} in \cite{frankel2015splitting}, the one incorporated in \cref{alg:palmi} is much more implementable. From \cite{bertsekas1999error}, we know that when $\h_i$ is linear (e.g., \Cref{exm:linear constraint}) or satisfies an enhanced version of the Slater constraint qualification (e.g., \Cref{exm:nonlinear constraint})
		$$\left\{\begin{array}{l}
			\exists~\hat\x_i\in\R^{m_i},~\st~\h_i(\hat\x_i)<0;~\text{and}\\
			\exists~\zeta\ge0,~\st~\dfrac{\snorm{\y_i-\hat\x_i}-\dist(\y_i,\calS_i)}{\min_{j=1,\ldots,p_i}\{-h_{i,j}(\hat\x_i)\}}\le\zeta,~\forall~\y_i\in\tilde\calS_i,
		\end{array}\right.$$
		the Hoffman-like bound in \Cref{lem:error bound} readily holds. We then could bound the solution errors without computing $\{\bar\z^{(k)}\}$. 
	\end{rem}
	
	\begin{rem}
		Since the inexact criterion described in \cref{alg:palmi} also covers the feasible inexactness, the theoretical results in this work apply to \palmf~as well. 
	\end{rem}

	\section{The Global Convergence Properties of \palmi}\label{sec:convergence of ipalm}
	
	\par In this section, we investigate the global convergence properties of \palmi, including the stationarity of any accumulation point and the iterate convergence. 
	
	\par In the beginning, we state some assumptions and conditions for $f$, $\{\calS_i\}_{i=1}^n$, $\{\mb{h}_i\}_{i=1}^n$, and $\{\eps^{(k)}\}$.
	
	\begin{assume}\label{assume:lip}
		The objective function $f$ in \cref{eqn:original prob} is Lipschitz continuously differentiable with respect to each variable block over $\otimes_{i=1}^n\bar\calS_i$, namely, for $i=1,\ldots,n$, there exists modulus $L_i>0$ such that, for any $\z_1$, $\z_2\in\otimes_{i=1}^n\bar\calS_i$, $\snorm{\nabla_if(\z_1)-\nabla_if(\z_2)}\le L_i\snorm{\z_1-\z_2}$.
	\end{assume}
	
	\begin{assume}\label{assume:compactness}
		For $i=1,\ldots,n$, $\calS_i$ is convex and compact and the following two hold for $\h_i$:
		\begin{enumerate}[label=(\alph*),topsep=0pt, parsep=0pt, itemsep=0pt]
			\item $\h_i$ is linear or satisfies the Slater constraint qualification;
			\item $\h_i$ satisfies the Hoffman-like error bound \cref{eqn:hoffman-like error bound}.
		\end{enumerate}
	\end{assume}
	
	\begin{cond}\label{cond:eps}
		\mbox{}
		\begin{enumerate}[label=(\alph*),topsep=0pt, parsep=0pt, itemsep=0pt]
			\item The sequence $\{\eps^{(k)}\}$ is nonnegative square summable.
			
			\item The sequence $\{\eps^{(k)}\}$ is nonnegative summable and there exists $\bar\theta\in(0,1)$ such that $\{(e^{(k)})^{\bar\theta}\}$ is summable, where $e^{(k)}:=\sum_{t=k}^\infty(\eps^{(t)})^2$ for any $k\ge0$.
		\end{enumerate}
	\end{cond}
	
	\begin{rem}\label{rem:choice of eps}
		\par One may find \Cref{cond:eps} (b) pretty restrictive at the first glance. In fact, given $\ell>1$, the sequence $\{\frac{\bar\eps}{(k+1)^{\ell}}\}$ just meets the demand. Note that, for this choice, $e^{(k)}$ decays as $\mathcal{O}(k^{-(2\ell-1)})$. To ensure the summability of $\{(e^{(k)})^{\bar\theta}\}$, it then suffices to choose $(0,1)\owns\bar\theta>\frac{1}{2\ell-1}$. We shall emphasize that what we only require is the existence of such $\bar\theta$ rather than its explicit value. 
		
		\par One more restrictive but more intuitive alternative for $\sum_{k=1}^\infty(e^{(k)})^{\bar\theta}<\infty$ is $\sum_{k=1}^\infty k(\eps^{(k)})^{2\bar\theta}<\infty$. Nonetheless, to retain potential flexibility, we use the one stated in \Cref{cond:eps} (b). 
	\end{rem}
	
	\par Let $L:=\max_iL_i$. We specify in \palmi~that $M_l=\gamma L$, where $\gamma>1$, and $M_u$ is any scalar not smaller than $M_l$. We further rewrite $M_u$ as $M$ for brevity. Some constants are defined beforehand: $\underline{\sigma}^{(k)}:=\min_i\sigma_i^{(k)}$, $\bar\sigma^{(k)}:=\max_i\sigma_i^{(k)}$, $\nu:=\frac{12}{\gamma-1}$, $\bar M:=\sqrt{3}(M+L\sqrt{n})$,
	\begin{equation}
		C_0^{(k)}:=\frac{\underline{\sigma}^{(k)}-L[1+\frac{6}{\nu}]}{2},~C_1^{(k)}:=\frac{\bar\sigma^{(k)}+L[\frac{\nu}{2}n^2+(2+\frac{2}{\nu}-\frac{\nu}{2})n+2\nu+\frac{4}{\nu}+3]}{2},
		\label{eqn:constants}
	\end{equation}
	and $\bar C_1:=2\omega^2\max_kC_1^{(k)}$, where $\omega$ is defined in \Cref{lem:error bound}. We use the notation ``$\Delta$'' for the difference between the optimal iterate and the real iterate; e.g.,
	\begin{equation*}
		\label{eqn:difference}
		\Delta\x_j^{(k+1)}:=\bar\x_j^{(k+1)}-\x_j^{(k+1)},\quad \Delta \z^{(k+1)}:=\bar \z^{(k+1)}-\z^{(k+1)}.
	\end{equation*}
	The proof sketch is as follows:
	\begin{enumerate}[label=(\roman*),topsep=0pt, parsep=0pt, itemsep=0pt]
		\item deducing the approximate sufficient reduction on the objective value sequence;

		\item deducing the sufficient reduction on the surrogate sequence;
		
		\item deducing the approximate relative error bound for subdifferential;
		
		\item showing the stationarity of any accumulation point;
		
		\item showing the iterate convergence with the help of the \L ojasiewicz property.
	\end{enumerate}
	
	\par We begin with a block-wise lemma. 
	
	\begin{lem}\label{lem:component decrease}
		Suppose \Cref{assume:lip} holds. Let $\{\z^{(k)}\}$ be the iterate sequence generated by \palmi. Then, for $i=1,\ldots,n$ and $k\ge0$,
		\begin{align}
			&f\big(\bar \x_{<i}^{(k+1)},\bar{\x}_i^{(k)},\bar\x_{>i}^{(k)}\big)-f\big(\bar \x_{<i}^{(k+1)},\bar{\x}_i^{(k+1)},\bar\x_{>i}^{(k)}\big)\nonumber\\
			\ge&\frac{\sigma_i^{(k)}-L[1+\frac{6}{\nu}]}{2}\snorm{\bar{\x}_i^{(k+1)}-\x_i^{(k)}}^2-\frac{\sigma_i^{(k)}+L[2\nu+3+\frac{4}{\nu}]}{2}\snorm{\Delta \x_i^{(k)}}^2\label{eqn:component sufficient decrease}\\
			&-\frac{L[\nu(i-1)+2+\frac{2}{\nu}]}{2}\snorm{\Delta\z^{(k+1)}}^2-\frac{L[\nu(n-i)+2+\frac{2}{\nu}]}{2}\snorm{\Delta\z^{(k)}}^2.\nonumber
		\end{align}
	\end{lem}
	
	\begin{proof}
		The proof mainly leverages \Cref{assume:lip} and the optimality of $\bar{\x}_i^{(k+1)}$ in \cref{eqn:block optimal}. We first note that the expression in the left-hand side of \cref{eqn:component sufficient decrease} can be splitted into five telescoping summations below:
		\begin{enumerate}[label=\text{part \arabic*},topsep=0pt, parsep=0pt, itemsep=0pt]
			\item $\sum_{j=1}^{i-1}f(\x_{<j}^{(k+1)},\bar\x_{[j,i)}^{(k+1)},\bar \x_{\ge i}^{(k)})-f(\x_{\le j}^{(k+1)},\bar\x_{(j,i)}^{(k+1)},\bar \x_{\ge i}^{(k)})$\label{forward 1};
			
			\item $\sum_{j=i+1}^nf(\x_{<i}^{(k+1)},\bar{\x}_i^{(k)},\x_{(i,j)}^{(k)},\bar\x_{\ge j}^{(k)})-f(\x_{<i}^{(k+1)},\bar{\x}_i^{(k)},\x_{(i,j]}^{(k)},\bar\x_{>j}^{(k)})$\label{forward 2};
			
			\item $f(\x_{<i}^{(k+1)},\bar{\x}_i^{(k)},\x_{>i}^{(k)})-f(\x_{<i}^{(k+1)},\bar{\x}_i^{(k+1)},\x_{>i}^{(k)})$\label{optimality};
			
			\item $\sum_{j=1}^{i-1}f(\bar \x_{<j}^{(k+1)},\x_{[j,i)}^{(k+1)},\bar{\x}_i^{(k+1)},\x_{>i}^{(k)})-f(\bar \x_{\le j}^{(k+1)},\x_{(j,i)}^{(k+1)},\bar{\x}_i^{(k+1)},\x_{>i}^{(k)})$\label{backward 1};
			
			\item $\sum_{j=i+1}^nf(\bar\x_{\le i}^{(k+1)},\bar \x_{(i,j)}^{(k)},\x_{\ge j}^{(k)})-f(\bar\x_{\le i}^{(k+1)},\bar \x_{(i,j]}^{(k)},\x_{>j}^{(k)})$\label{backward 2}.
		\end{enumerate}
		
		By \Cref{assume:lip} and the optimality of $\bar{\x}_i^{(k+1)}$ in \cref{eqn:prox-linear subprob new}, we readily have a lower bound for \cref{optimality}:
		\begin{align}
			\cref{optimality}\ge&-\inner{\nabla_if(\x_{<i}^{(k+1)},\bar{\x}_i^{(k)},\x_{>i}^{(k)}),\bar{\x}_i^{(k+1)}-\bar{\x}_i^{(k)}}-\frac{L_i}{2}\snorm{\bar{\x}_i^{(k+1)}-\bar{\x}_i^{(k)}}^2\nonumber\\
			=&-\inner{\nabla_if(\x_{<i}^{(k+1)},\x_{\ge i}^{(k)}),\bar{\x}_i^{(k+1)}-\bar{\x}_i^{(k)}}-\frac{L_i}{2}\snorm{\bar{\x}_i^{(k+1)}-\bar{\x}_i^{(k)}}^2\nonumber\\
			&+\inner{\nabla_if(\x_{<i}^{(k+1)},\x_{\ge i}^{(k)})-\nabla_if(\x_{<i}^{(k+1)},\bar{\x}_i^{(k)},\x_{>i}^{(k)}),\bar{\x}_i^{(k+1)}-\bar{\x}_i^{(k)}}\nonumber\\
			\ge&\frac{\sigma_i^{(k)}}{2}[\snorm{\bar{\x}_i^{(k+1)}-\x_i^{(k)}}^2-\snorm{\Delta \x_i^{(k)}}^2]-\frac{L}{2}\lrbracket{1+\frac{1}{\nu}}\snorm{\bar\x_i^{(k+1)}-\x_i^{(k)}}^2-\frac{L[1+\nu]}{2}\snorm{\Delta\x_i^{(k)}}^2\label{eqn:optimality lb}\\
			&-\frac{L}{2}\lrsquare{(\nu+2)\snorm{\Delta\x_i^{(k)}}^2+\frac{1}{\nu}\snorm{\bar\x_i^{(k+1)}-\x_i^{(k)}}^2}\nonumber\\
			=&\frac{\sigma_i^{(k)}-L[1+\frac{2}{\nu}]}{2}\snorm{\bar\x_i^{(k+1)}-\x_i^{(k)}}^2-\frac{\sigma_i^{(k)}+L[2\nu+3]}{2}\snorm{\Delta\x_i^{(k)}}^2,\nonumber
		\end{align}
		where the second inequality also invokes \Cref{lem:inequality lemma} (i) and the definition of $L$. 
		
		\par Since the analyses for the remaining differences are analogous, we merely demonstrate in detail for \cref{forward 1,backward 1}. By \Cref{assume:lip},
		\begin{align*}
			\text{\cref{forward 1}}\ge&\sum_{j=1}^{i-1}\left[\inner{\nabla_jf(\x_{<j}^{(k+1)},\bar\x_{[j,i)}^{(k+1)},\bar \x_{\ge i}^{(k)}),-\Delta\x_j^{(k+1)}}-\frac{L_j}{2}\snorm{\Delta\x_j^{(k+1)}}^2\right],\\
			\text{\cref{backward 1}}\ge&\sum_{j=1}^{i-1}\left[\inner{\nabla_jf(\bar \x_{<j}^{(k+1)},\x_{[j,i)}^{(k+1)},\bar{\x}_i^{(k+1)},\x_{>i}^{(k)}),\Delta\x_j^{(k+1)}}-\frac{L_j}{2}\snorm{\Delta\x_j^{(k+1)}}^2\right].
		\end{align*}
		Combining the above two implies
		\begin{align*}
			&\text{\cref{forward 1}}+\text{\cref{backward 1}}
			\ge-L\sum_{j=1}^{i-1}\snorm{\Delta\x_j^{(k+1)}}^2-L\sum_{j=1}^{i-1}\norm{\begin{pmatrix}
					\x_{<j}^{(k+1)} - \bar\x_{<j}^{(k+1)}\\
					\bar\x_{[j,i)}^{(k+1)} - \x_{[j,i)}^{(k+1)}\\
					\bar\x_i^{(k)} - \bar\x_i^{(k+1)}\\
					\bar\x_{>i}^{(k)} - \x_{>i}^{(k)}
			\end{pmatrix}}\snorm{\Delta\x_j^{(k+1)}}\nonumber\\
			=&-L\sum_{j=1}^{i-1}\snorm{\Delta\x_j^{(k+1)}}^2-L\norm{\begin{pmatrix}
					\x_{<i}^{(k+1)} - \bar\x_{<i}^{(k+1)}\\
					\bar\x_i^{(k)} - \bar\x_i^{(k+1)}\\
					\bar\x_{>i}^{(k)} - \x_{>i}^{(k)}
			\end{pmatrix}}\sum_{j=1}^{i-1}\snorm{\Delta\x_j^{(k+1)}}\\
			\ge&-L\sum_{j=1}^{i-1}\snorm{\Delta\x_j^{(k+1)}}^2-L\norm{\begin{pmatrix}
					\x_{<i}^{(k+1)} - \bar\x_{<i}^{(k+1)}\\
					\bar\x_i^{(k)} - \bar\x_i^{(k+1)}\\
					\bar\x_{>i}^{(k)} - \x_{>i}^{(k)}
			\end{pmatrix}}\sqrt{(i-1)\sum_{j=1}^{i-1}\snorm{\Delta\x_j^{(k+1)}}^2}\\
			\ge&-\frac{L[\nu(i-1)+2]}{2}\sum_{j=1}^{i-1}\snorm{\Delta\x_j^{(k+1)}}^2-\frac{L}{\nu}\lrsquare{\snorm{\bar\x_i^{(k+1)}-\x_i^{(k)}}^2+\snorm{\Delta\x_i^{(k)}}^2}\\
			&-\frac{L}{2\nu}\lrsquare{\sum_{l<i}\snorm{\Delta\x_l^{(k+1)}}^2+\sum_{l>i}\snorm{\Delta\x_l^{(k)}}^2},
		\end{align*}
		where the first inequality comes from \Cref{assume:lip} and the definition of $L$, the second and the last inequality follow from \Cref{lem:inequality lemma} (i). Similar arguments yield a lower bound for \cref{forward 2}~$+$~\cref{backward 2}:
		\begin{align}
			&\text{\cref{forward 2}}+\text{\cref{backward 2}}\nonumber\\
			\ge&-\frac{L[\nu(n-i)+2]}{2}\sum_{j=i+1}^{n}\snorm{\Delta\x_j^{(k)}}^2-\frac{L}{\nu}\lrsquare{\snorm{\bar{\x}_i^{(k+1)}-\x_i^{(k)}}^2+\snorm{\Delta \x_i^{(k)}}^2}\nonumber\\
			&-\frac{L}{2\nu}\lrsquare{\sum_{l<i}\snorm{\Delta \x_l^{(k+1)}}^2+\sum_{l>i}\snorm{\Delta \x_l^{(k)}}^2}.\nonumber
		\end{align}
		
		\par Combining \cref{eqn:optimality lb} with the last two inequalities, we have
		\begin{align*}
			&f\lrbracket{\bar \x_{<i}^{(k+1)},\bar{\x}_i^{(k)},\bar\x_{>i}^{(k)}}-f\lrbracket{\bar \x_{<i}^{(k+1)},\bar{\x}_i^{(k+1)},\bar\x_{>i}^{(k)}}\\
			\ge&\frac{\sigma_i^{(k)}-L[1+\frac{6}{\nu}]}{2}\snorm{\bar{\x}_i^{(k+1)}-\x_i^{(k)}}^2-\frac{\sigma_i^{(k)}+L[2\nu+3+\frac{4}{\nu}]}{2}\snorm{\Delta \x_i^{(k)}}^2\nonumber\\
			&-\frac{L[\nu(i-1)+2+\frac{2}{\nu}]}{2}\sum_{l<i}\snorm{\Delta\x_l^{(k+1)}}^2-\frac{L[\nu(n-i)+2+\frac{2}{\nu}]}{2}\sum_{l>i}\snorm{\Delta\x_l^{(k)}}^2,\nonumber
		\end{align*}
		which completes the proof after noticing the definition of $\z^{(k+1)}$.
	\end{proof}
	
	\par The following approximate sufficient reduction is then a direct corollary.
	
	\begin{prop}\label{prop:approx suff reduce}
		Suppose \Cref{assume:lip} holds. Let $\{\z^{(k)}\}$ be the iterate sequence generated by \palmi. Then, for any $k\ge0$, 
		\begin{equation}
			f(\bar \z^{(k)})-f(\bar \z^{(k+1)})\ge C_0^{(k)}\snorm{\bar \z^{(k+1)}-\z^{(k)}}^2-C_1^{(k)}\snorm{\Delta \z^{(k)}}^2-C_1^{(k+1)}\snorm{\Delta \z^{(k+1)}}^2.
			\label{eqn:approx suff reduce}
		\end{equation}
	\end{prop}
	
	\begin{proof}
		From \Cref{lem:component decrease}, we have by telescoping summation
		\begin{align}
			&f(\bar \z^{(k)})-f(\bar \z^{(k+1)})\nonumber\\
			\ge&\sum_{i=1}^n\left[f\big(\bar \x_{<i}^{(k+1)},\bar{\x}_i^{(k)},\bar\x_{>i}^{(k)}\big)-f\big(\bar \x_{<i}^{(k+1)},\bar{\x}_i^{(k+1)},\bar\x_{>i}^{(k)}\big)\right]\nonumber\\
			\ge&\sum_{i=1}^n\left[\frac{\sigma_i^{(k)}-L[1+\frac{6}{\nu}]}{2}\snorm{\bar{\x}_i^{(k+1)}-\x_i^{(k)}}^2-\frac{\sigma_i^{(k)}+L[2\nu+3+\frac{4}{\nu}]}{2}\snorm{\Delta \x_i^{(k)}}^2\right.\nonumber\\
			&\left.-\frac{L[\nu(i-1)+2+\frac{2}{\nu}]}{2}\snorm{\Delta\z^{(k+1)}}^2-\frac{L[\nu(n-i)+2+\frac{2}{\nu}]}{2}\snorm{\Delta\z^{(k)}}^2\right]\nonumber\\
			\ge&\frac{\min_i\sigma_i^{(k)}-L[1+\frac{6}{\nu}]}{2}\snorm{\bar \z^{(k+1)}-\z^{(k)}}^2-\frac{L[\frac{\nu}{2}n^2+(2+\frac{2}{\nu}-\frac{\nu}{2})n]}{2}\snorm{\Delta\z^{(k+1)}}^2\nonumber\\
			&-\frac{\max_i\sigma_i^{(k)}+L[\frac{\nu}{2}n^2+(2+\frac{2}{\nu}-\frac{\nu}{2})n+2\nu+\frac{4}{\nu}+3]}{2}\snorm{\Delta \z^{(k)}}^2,\nonumber
		\end{align}
		which completes the proof by noting the definition of $\underline{\sigma}^{(k)}$, $\bar\sigma^{(k)}$, $C_0^{(k)}$, and $C_1^{(k)}$ in \cref{eqn:constants}.
	\end{proof}
	
	\par The infeasibility brings additional error terms, in particular, the error term from the last step, to the right-hand side of \cref{eqn:approx suff reduce}. Consequently, the nonmonotonicity of \palmi~appears to be inevitable.
	
	\par Instead of striving to achieve monotonicity, in the spirit of \cite{li2021convergence,sun2017convergence,yang2017proximal}, we explicitly include the error terms in a surrogate sequence for which a sufficient reduction result is obtained. 
	
	\begin{prop}\label{prop:suff reduce}
		Suppose \Cref{assume:lip,assume:compactness} holds. Let $\{\z^{(k)}\}$ be the iterate sequence generated by \palmi, where $\{\eps^{(k)}\}$ fulfills \Cref{cond:eps} (a).  Then the following assertions hold.
		\begin{enumerate}[label=(\roman*),topsep=0pt, parsep=0pt, itemsep=0pt]
			\item The sequence $\big\{v^{(k)}:=F(\bar \z^{(k)})+u^{(k)}+u^{(k+1)}\big\}$ is well defined, where $u^{(k)}:=\sum_{t=k}^\infty C_1^{(t)}\snorm{\Delta \z^{(t)}}^2$.
			
			\item For any $k\ge0$, $v^{(k)}-v^{(k+1)}\ge C_0^{(k)}\norm{\bar \z^{(k+1)}-\z^{(k)}}^2\ge0$.
			
			\item The sequence $\{v^{(k)}\}$ converges monotonically to some $\bbar F\ge0$, which is attainable for $F$ over $\otimes_{i=1}^n\calS_i$. In particular, $F(\bar \z^{(k)})\to\bbar F$ as $k\to\infty$. 
			
			\item If there exists an integer $\tilde k\in\N$ such that $v^{(\tilde k)}=\bbar F$, then $v^{(k)}=\bbar F$ and $\z^{(k)}=\bar\z^{(k+1)}$ for any $k\ge\tilde k$. Moreover, if there exists an integer $\hat k\ge\tilde k$ such that $\z^{(\hat k)}=\bar\z^{(\hat k)}$, then one further has $\z^{(k)}=\bar\z^{(k+1)}=\z^{(k+1)}$ for any $k\ge\hat k$. 
		\end{enumerate}
	\end{prop}
	\begin{proof}
		(i) Since $\{\eps^{(k)}\}$ is square summable, for each $k$, 
		$$u^{(k)}=\sum_{t=k}^\infty C_1^{(t)}\snorm{\Delta \z^{(t)}}^2\le\omega^2\sum_{t=k}^\infty C_1^{(t)}(\eps^{(t)})^2\le\frac{\bar C_1}{2}\sum_{t=k}^\infty(\eps^{(t)})^2<\infty,$$
		where the first inequality follows from \Cref{lem:error bound}. The well-definedness then follows.
		
		\vskip 0.1cm
		
		(ii) Simply plugging the definition of $\{v^{(k)}\}$ into \Cref{prop:approx suff reduce} leads to the first inequality. The second inequality is due to $\sigma_i^{(k)}\ge\gamma L$ in \palmi, which yields $C_0^{(k)}\ge\frac{\gamma-1}{4}L$.
		
		\vskip 0.1cm
		
		(iii) Since $\{\eps^{(k)}\}$ is square summable, we have $u^{(k)}\to0$ as $k\to\infty$. The desired result is then obtained from the sufficient reduction given by (ii), $C_0^{(k)}\ge\frac{\gamma-1}{4}L$ for any $k\ge0$, the lower boundedness and continuity of $F$ over $\otimes_{i=1}^n\calS_i$.
		
		\vskip 0.1cm
		
		(iv) The former part follows directly from statements (ii), (iii), and the fact that $C_0^{(k)}\ge\frac{\gamma-1}{4}L$ for any $k\ge0$. To show the latter part, recalling the definition of $v^{(k)}$, we obtain from $v^{(\hat k)}=v^{(\hat k+1)}$ that
		$$v^{(\hat k)}=F(\bar\z^{(\hat k)})+u^{(\hat k)}+u^{(\hat k+1)}=F(\bar\z^{(\hat k+1)})+u^{(\hat k+1)}+u^{(\hat k+2)}=v^{(\hat k+1)},$$
		which, combined with $\bar\z^{(\hat k+1)}=\z^{(\hat k)}=\bar\z^{(\hat k)}$, yields 
		$$0=u^{(\hat k)}-u^{(\hat k+2)}=C_0^{(\hat k)}\snorm{\Delta\z^{(\hat k)}}+C_0^{(\hat k+1)}\snorm{\Delta\z^{(\hat k+1)}}=C_0^{(\hat k+1)}\snorm{\Delta\z^{(\hat k+1)}}.$$
		Since $C_0^{(\hat k+1)}\ge\frac{\gamma-1}{4}L>0$, we have $\bar\z^{(\hat k+1)}=\z^{(\hat k+1)}$. By induction, we could derive the desired relation for any $k\ge\hat k$. 
	\end{proof}
	
	\par Next, we seek to prove the approximate relative error bound for subdifferential. 
	
	\begin{prop}\label{prop:approx subgrad lb}
		Suppose \Cref{assume:lip} holds. Let $\{\z^{(k)}\}$ be the iterate sequence generated by \palmi. Then, for any $k\ge0$, there exists $\w^{(k+1)}\in\partial F(\bar \z^{(k+1)})$ such that,
		\begin{equation*}
			\snorm{\w^{(k+1)}}\le\bar M\lrsquare{\snorm{\bar \z^{(k+1)}-\z^{(k)}}+\snorm{\Delta \z^{(k+1)}}}.
		\end{equation*}
	\end{prop}
	
	\begin{proof}
		For each $i\in\{1,\ldots,n\}$, it follows from the optimality of $\bar{\x}_i^{(k+1)}$ in \cref{eqn:block optimal} and \Cref{rem:subdiff} (iii) that there exists $\mb{a}_i^{(k+1)}\in \partial\delta_{\calS_i}(\bar{\x}_i^{(k+1)})$ such that
		\begin{equation*}
			0=\nabla_if\big(\x_{<i}^{(k+1)},\bar{\x}_i^{(k+1)},\x_{>i}^{(k)}\big)+\sigma_i^{(k)}\big(\bar{\x}_i^{(k+1)}-\x^{(k)}\big)+\mb{a}_i^{(k+1)}.
		\end{equation*}
		Using the above relation and the calculus of Fr\'echet subdifferential, we have
		\begin{align*}
			\partial F(\bar \z^{(k+1)})&\owns\nabla f(\bar \z^{(k+1)})+\big(\mb{a}_i^{(k+1)}\big)_{i=1}^n\\
			&=\lrbracket{\nabla_if(\bar \z^{(k+1)})-\nabla_if(\x_{<i}^{(k+1)},\bar{\x}_i^{(k+1)},\x_{>i}^{(k)})-\sigma_i^{(k)}(\bar{\x}_i^{(k+1)}-\x_i^{(k)})}_{i=1}^n.
		\end{align*}
		
		Let $\w^{(k+1)}:=\nabla f(\bar \z^{(k+1)})+\big(\mb{a}_i^{(k+1)}\big)_{i=1}^n$. We have, for $i=1,\ldots,n$,
		\begin{align}
			&\big\Vert\nabla_if(\bar \z^{(k+1)})-\nabla_if(\x_{<i}^{(k+1)},\bar{\x}_i^{(k+1)},\x_{>i}^{(k)})-\sigma_i^{(k)}(\bar{\x}_i^{(k+1)}-\x_i^{(k)})\big\Vert\nonumber\\
			\le&L\norm{\begin{pmatrix}
					\bar\x_{<i}^{(k+1)}-\x_{<i}^{(k+1)}\\
					\bar\x_{>i}^{(k+1)}-\x_{>i}^{(k)}
			\end{pmatrix}}+\sigma_i^{(k)}\snorm{\bar{\x}_i^{(k+1)}-\x_i^{(k)}}\nonumber\\
			\le&L\lrsquare{\snorm{\Delta\z^{(k+1)}}+\snorm{\bar\z^{(k+1)}-\z^{(k)}}}+\sigma_i^{(k)}\snorm{\bar{\x}_i^{(k+1)}-\x_i^{(k)}},\nonumber
		\end{align}
		where the first inequality follows from \Cref{assume:lip} and the definition of $L$. Therefore, it follows again from \Cref{lem:inequality lemma} (i) that
		\begin{align}
			\snorm{\w^{(k+1)}}^2&=\sum_{i=1}^n\lrsquare{L\snorm{\Delta\z^{(k+1)}}+L\snorm{\bar\z^{(k+1)}-\z^{(k)}}+\sigma_i^{(k)}\snorm{\bar\x_i^{(k+1)}-\x_i^{(k)}}}^2\nonumber\\
			&\le3nL^2\lrsquare{\snorm{\Delta\z^{(k+1)}}^2+\snorm{\bar\z^{(k+1)}-\z^{(k)}}^2}+3(\bar\sigma^{(k)})^2\snorm{\z^{(k+1)}-\z^{(k)}}^2,\nonumber
		\end{align}
		which completes the proof by recalling the definition of $\bar M$ ahead of \cref{eqn:constants}.
	\end{proof}
	
	\par In what follows, we prove a weak result, i.e., the stationarity of any accumulation point of $\{\z^{(k)}\}$, assuming that $\{\eps^{(k)}\}$ is square summable. The accumulation point set $\Omega(\z^{(0)})$ is defined as
	\begin{equation*}
		\Omega(\z^{(0)}):=\lrbrace{\z=(\x_i)_{i=1}^n\in\otimes_{i=1}^n\R^{m_i}:~\exists~\mathcal{K}\subseteq\N,~\st~\z^{(k)}\to\z\text{ in }\mathcal{K}}.
	\end{equation*}
	Given $\eps^{(k)}\to0$, one has from \Cref{lem:error bound} that $\snorm{\Delta\z^{(k)}}\to0$, which yields
	\begin{equation}
		\Omega(\z^{(0)})=\lrbrace{\z=(\x_i)_{i=1}^n\in\otimes_{i=1}^n\R^{m_i}:~\exists~\mathcal{K}\subseteq\N,~\st~\bar\z^{(k)}\to\z\text{ in }\mathcal{K}}.
		\label{eqn:omega identity}
	\end{equation}
	
	\begin{prop}\label{prop:properties of acc pt set}
		Suppose \Cref{assume:compactness,assume:lip} hold. Let $\{\z^{(k)}\}$ be the iterate sequence generated by \palmi, where $\{\eps^{(k)}\}$ fulfills \Cref{cond:eps} (a). Then we have that
		\begin{enumerate}[label=(\roman*),topsep=0pt, parsep=0pt, itemsep=0pt]
			\item $\Omega(\z^{(0)})\subseteq\otimes_{i=1}^n\calS_i$ is nonempty and each element is a stationary point of $F$;\label{lem:acc:KKT}
			
			\item $\Omega(\z^{(0)})$ is compact connected and $\dist(\bar \z^{(k)},\Omega(\z^{(0)}))\to0$;\label{lem:acc:compact}
			
			\item $F$ is finite and constant on $\Omega(\z^{(0)})$.\label{lem:acc:finite and constant} 
		\end{enumerate}
	\end{prop}
	
	\begin{proof}
		\par (i) Since $\{\bar \z^{(k)}\}\subseteq\otimes_{i=1}^n\calS_i$ and $\calS_i$ is bounded (by \Cref{assume:compactness}), there exist a subsequence $\{\bar \z^{(k)}\}_{k\in\mathcal{K}}$ and $\bar\z$ such that $\bar \z^{(k)}\to \bar\z$ in $\mathcal{K}$, which gives $\Omega(\z^{(0)})\ne\emptyset$ in view of \cref{eqn:omega identity}. Also note that $\otimes_{i=1}^n\calS_i$ is closed (by \Cref{assume:compactness}), thus $\bar\z\in\otimes_{i=1}^n\calS_i$ and hence $\Omega(\z^{(0)})\subseteq\otimes_{i=1}^n\calS_i$. 
		
		\par From \Cref{prop:suff reduce} (ii) and the fact that $C_0^{(k)}\ge\frac{\gamma-1}{4}L$ for any $k\ge0$, we get, for any $k\ge1$,
		$$\frac{\gamma-1}{4}L\snorm{\bar \z^{(k+1)}-\z^{(k)}}^2\le C_0^{(k)}\snorm{\bar \z^{(k+1)}-\z^{(k)}}^2\le v^{(k)}-v^{(k+1)}.$$
		Summing the above inequality over $k$ from $r$ to $s$ ($s\ge r>1$), we have
		\begin{align}
			&\frac{\gamma-1}{4}L\sum_{k=r}^s\snorm{\bar \z^{(k+1)}-\z^{(k)}}^2\le\sum_{k=r}^s(v^{(k)}-v^{(k+1)})=v^{(r)}-v^{(s+1)}\le v^{(r)}-\bbar F\nonumber\\
			\le&(F(\bar\z^{(r)})-\bbar F)+2\sum_{t=r}^\infty C_1^{(t)}\snorm{\Delta\z^{(t)}}^2\le(F(\bar\z^{(r)})-\bbar F)+\bar C_1\sum_{t=r}^\infty(\eps^{(k)})^2,\nonumber
		\end{align}
		where the second inequality follows from \Cref{prop:suff reduce} (iii) and the last one is due to \cref{eqn:constants} and \Cref{lem:error bound}. By the square summability of $\{\eps^{(k)}\}$, the term in the right-hand side is finite for any $s$. Therefore $\Vert \bar \z^{(k+1)}-\z^{(k)}\Vert\to0$. \Cref{prop:approx subgrad lb}, combined with \Cref{lem:error bound} and the square summability of $\{\eps^{(k)}\}$, further implies $\w^{(k+1)}\to0$.
		
		\par Now pick $\bar\z\in\Omega(\z^{(0)})$ and the associated converging subsequence $\{\bar \z^{(k+1)}\}_{k\in\mathcal{K}}$. Since the subsequence $\{(\bar \z^{(k+1)},\w^{(k+1)})\}_{k\in\mathcal{K}}$ converges to $\lrbracket{\bar\z,0}$ and $F(\bar \z^{(k+1)})\to \bbar F=F(\bar\z)$ (by \Cref{prop:suff reduce} (iii) and the continuity of $F$ in its domain), we deduce from \Cref{rem:subdiff} that $0\in\partial F(\bar\z)$. We complete this item by the arbitrariness of $\bar\z$. 
		
		\vskip 0.1cm 
		
		\par (ii) \& (iii) We directly follow from \cite[Lemma 5]{bolte2014proximal2}.
	\end{proof}

	\par To obtain the iterate convergence of \palmi~when solving the general nonconvex problem \cref{eqn:extended form}, we assume the \L ojasiewicz property for $F$ in \cref{eqn:extended form} and impose stronger assumptions on $\{\eps^{(k)}\}$.
	
	\begin{assume}\label{assume:KL}
		The \L ojasiewicz property holds for the objective function $F$ in \cref{eqn:extended form} at each stationary point. 
	\end{assume}
	
	\par By \Cref{lem:KL}, \Cref{prop:properties of acc pt set}, and \Cref{assume:KL}, $F$ satisfies the uniformized \L ojasiewicz property over $\Omega(\z^{(0)})$, provided the square summability of $\{\eps^{(k)}\}$. Let $c$, $\eta>0$, and $\theta\in[0,1)$ be the scalars associated with $\Omega=\Omega(\z^{(0)})$ and $G=F$. 
	
	\begin{thm}\label{thm:global convergence}
		Suppose \Cref{assume:compactness,assume:lip,assume:KL} hold. Let $\{\z^{(k)}\}$ be the iterate sequence generated by \palmi, where $\{\eps^{(k)}\}$ fulfills \Cref{cond:eps} (b). Then the sequence $\{\z^{(k)}\}$ converges to a stationary point of $F$.
	\end{thm}
	\begin{proof}
		\par Without loss of generality, we assume that $\bar\theta\ge\theta$ in \Cref{cond:eps} (b), i.e., $\bar\theta$ is lifted from $\theta$. To show the convergence of $\{\z^{(k)}\}$, it suffices to prove that $\{\snorm{\z^{(k)}-\z^{(k+1)}}\}$ has finite length. Note that, for any $s$, $t\in\N$ ($s\ge t$),
		\begin{equation}
			\sum_{k=s}^t\snorm{\z^{(k)}-\z^{(k+1)}}\le\sum_{k=s}^t\snorm{\bar\z^{(k+1)}-\z^{(k)}}+\sum_{k=s}^t\snorm{\Delta\z^{(k+1)}},
			\label{eqn:finite length of sequence}
		\end{equation}
		and the assumption $\sum_{k=1}^\infty\eps^{(k)}<\infty$ already guarantees the finiteness of the second term regardless of $t$. Hence what remains is to show the summability of $\{\snorm{\bar\z^{(k+1)}-\z^{(k)}}\}$. We divide the discussion into two cases.
		
		\vskip 0.1cm
		
		\par\noindent\textbf{Case I}. \textit{There exists an integer $\tilde k\in\N$ such that $v^{(\tilde k)}-\bbar F=0$.}\\[0.1cm]
		By \Cref{prop:suff reduce} (iv), we could deduce $v^{(k)}=\bbar F$ and $\z^{(k)}=\bar\z^{(k+1)}$ for any $k\ge\tilde k$, giving rise directly to the summability of $\{\snorm{\bar\z^{(k+1)}-\z^{(k)}}\}$. Based upon \cref{eqn:finite length of sequence} and \Cref{prop:properties of acc pt set} (i), $\{\z^{(k)}\}$ converges a stationary point of $F$. 
		
		\par In this case, if there exists an integer $\hat k\in\N$ ($\hat k\ge\tilde k$) such that $\z^{(\hat k)}=\bar\z^{(\hat k)}$, we could infer something better. Indeed, by \Cref{prop:suff reduce} (iv), one has $\z^{(k)}=\bar\z^{(k+1)}=\z^{(k+1)}$ for any $k\ge\hat k$. \Cref{prop:approx subgrad lb} then implies $\w^{(k+1)}=0$ for all $k\ge\hat k$. That is to say, $\{\z^{(k)}\}$ finitely terminates at a stationary point of $F$. 
		
		\vskip 0.1cm	
		
		\par\noindent\textbf{Case II}. \textit{For any $k\ge0$, $v^{(k)}-\bbar F>0$.}\\[0.1cm]
		We prove by showing a recursive relationship for $\{\snorm{\bar\z^{(k+1)}-\z^{(k)}}\}$. Let $\bar\varphi(s):=\frac{c}{1-\bar\theta}s^{1-\bar\theta}$. Since $v^{(k)}-\bbar F$ is positive, $\bar\varphi'(v^{(k)}-\bbar F)$ is well defined. By the concavity of $\bar\varphi$, we could derive that
		\begin{align}
			&\D^{(k,k+1)}:=\bar\varphi(v^{(k)}-\bbar F)-\bar\varphi(v^{(k+1)}-\bbar F)\ge\bar\varphi'(v^{(k)}-\bbar F)(v^{(k)}-v^{(k+1)})\nonumber\\
			\ge&c\frac{C_0^{(k)}\snorm{\bar \z^{(k+1)}-\z^{(k)}}^2}{\lrsquare{F(\bar \z^{(k)})-\bbar F+u^{(k)}+u^{(k+1)}}^{\bar\theta}}\ge\frac{c(\gamma-1)L}{4}\frac{\snorm{\bar \z^{(k+1)}-\z^{(k)}}^2}{\lrsquare{F(\bar \z^{(k)})-\bbar F+u^{(k)}+u^{(k+1)}}^{\bar\theta}},\nonumber
		\end{align}
		where the second inequality is due to \Cref{prop:suff reduce} (ii) and the definition of $v^{(k)}$, and the last one follows from $C_0^{(k)}\ge\frac{\gamma-1}{4}L$ for any $k\ge0$. The above inequality further provides an upper bound:
		\begin{align}
			&\snorm{\bar \z^{(k+1)}-\z^{(k)}}\nonumber\\
			\le&2\sqrt{\frac{1}{c(\gamma-1)L}}\sqrt{\lrsquare{F(\bar \z^{(k)})-\bbar F+u^{(k)}+u^{(k+1)}}^{\bar\theta}\cdot\D^{(k,k+1)}}\label{eqn:ub for iter gap}\\
			\le&\sqrt{\frac{1}{c(\gamma-1)L}}\left[\frac{1}{p}\lrsquare{F(\bar \z^{(k)})-\bbar F+u^{(k)}+u^{(k+1)}}^{\bar\theta}+p\D^{(k,k+1)}\right],\nonumber
		\end{align}
		where the second inequality comes from \Cref{lem:inequality lemma} (i). The constant $p$ is chosen such that 
		\begin{equation}
			p>\sqrt{\frac{c}{(\gamma-1)L}}\bar M.
			\label{eqn:p}
		\end{equation}
		On the other hand, by \Cref{lem:inequality lemma} (iii), we obtain 
		\begin{align}
			&\lrsquare{F(\bar \z^{(k)})-\bbar F+u^{(k)}+u^{(k+1)}}^{\bar\theta}\nonumber\\
			\le&\abs{F(\bar \z^{(k)})-\bbar F}^{\bar\theta}+(2u^{(k)})^{\bar\theta}=\abs{F(\bar \z^{(k)})-\bbar F}^{\bar\theta}+\big(2\sum_{t=k}^\infty C_1^{(t)}\snorm{\Delta\z^{(t)}}^2\big)^{\bar\theta}\label{eqn:df ub}\\
			\le&\abs{F(\bar \z^{(k)})-\bbar F}^{\bar\theta}+\bar C_1^{\bar\theta}\big(\sum_{t=k}^\infty(\eps^{(t)})^2\big)^{\bar\theta}=\abs{F(\bar \z^{(k)})-\bbar F}^{\bar\theta}+\bar C_1^{\bar\theta}(e^{(k)})^{\bar\theta},\nonumber
		\end{align}
		where the second inequality is from \cref{eqn:constants} and \Cref{lem:error bound}. From the summability of $\{\eps^{(k)}\}$, \Cref{prop:suff reduce} (iii), and \Cref{prop:properties of acc pt set} (ii), we know that there exists an integer $k_1\in\N$ such that, for all $k\ge k_1$, $\bar \z^{(k)}\in\{\z:\dist(\z,\Omega(\z^{(0)}))<\eta\}\bigcap\{\z:\abs{F(\z)-\bbar F}<1\}$. \Cref{lem:KL}, \Cref{coro:larger theta}, and \Cref{prop:approx subgrad lb} then yield that, for any $k\ge k_1$,
		\begin{equation}
			c\bar M[\snorm{\bar\z^{(k)}-\z^{(k-1)}}+\snorm{\Delta\z^{(k)}}]\ge c\snorm{\w^{(k)}}\ge c\cdot\dist(0,\partial F(\bar \z^{(k)}))\ge\abs{F(\bar \z^{(k)})-\bbar F}^{\bar\theta}.
			\label{eqn:kl apply}
		\end{equation}
		Plugging \cref{eqn:kl apply} into \cref{eqn:df ub} and invoking \Cref{prop:approx subgrad lb}, one has, for any $k\ge k_1$, 
		\begin{equation}
			\label{eqn:df ub2}
			\lrsquare{F(\bar \z^{(k)})-\bbar F+u^{(k)}+u^{(k+1)}}^{\bar\theta}
			\le c\bar M[\snorm{\bar\z^{(k)}-\z^{(k-1)}}+\snorm{\Delta\z^{(k)}}]+\bar C_1^{\bar\theta}(e^{(k)})^{\bar\theta}.
		\end{equation}
		Inserting \cref{eqn:df ub2} into \cref{eqn:ub for iter gap} gives, for any $k\ge k_1$,
		\begin{equation*}
			\sqrt{c(\gamma-1)L}\snorm{\bar \z^{(k+1)}-\z^{(k)}}\le\frac{c\bar M}{p}[\snorm{\bar\z^{(k)}-\z^{(k-1)}}+\snorm{\Delta\z^{(k)}}]+\frac{\bar C_1^{\bar\theta}}{p}(e^{(k)})^{\bar\theta}+p\D^{(k,k+1)}.\nonumber
		\end{equation*}
		
		\par Summing the above inequality over $k$ from $s$ to $t$ ($t\ge s\ge k_1)$ yields 
		\begin{align}
			&\sqrt{c(\gamma-1)L}\sum_{k=s}^t\snorm{\bar \z^{(k+1)}-\z^{(k)}}\nonumber\\
			\le&\frac{c\bar M}{p}\lrsquare{\sum_{k=s-1}^t\snorm{\bar\z^{(k+1)}-\z^{(k)}}+\sum_{k=s}^t\snorm{\Delta\z^{(k)}}}+\frac{\bar C_1^{\bar\theta}}{p}\sum_{k=s}^t(e^{(k)})^{\bar\theta}+p\D^{(s,t+1)}\nonumber\\
			\le&\frac{c\bar M}{p}\lrsquare{\sum_{k=s-1}^t\snorm{\bar\z^{(k+1)}-\z^{(k)}}+\omega\sum_{k=1}^\infty\eps^{(k)}}+\frac{\bar C_1^{\bar\theta}}{p}\sum_{k=1}^\infty (e^{(k)})^{\bar\theta}+p\bar\varphi(v^{(s)}-\bbar F),\nonumber
		\end{align}
		where the second inequality follows from \Cref{lem:error bound}, the assumptions on $\{\eps^{(k)}\}$, and $\bar\varphi>0$ over $(0,\infty)$. Hence
		\begin{align}
			&\lrsquare{\sqrt{c(\gamma-1)L}-\frac{c\bar M}{p}}\sum_{k=s}^t\snorm{\bar \z^{(k+1)}-\z^{(k)}}\label{eqn:finite length}\\
			\le&\frac{c\bar M}{p}\lrsquare{\snorm{\bar\z^{(s)}-\z^{(s-1)}}+\omega\sum_{k=1}^\infty\eps^{(k)}}+\frac{\bar C_1^{\bar\theta}}{p}\sum_{k=1}^\infty (e^{(k)})^{\bar\theta}+p\bar\varphi(v^{(s)}-\bbar F).\nonumber
		\end{align}
		Note that due to \cref{eqn:p}, the coefficient in the left-hand side of \cref{eqn:finite length} is positive. Therefore, in view of the assumptions on $\{\eps^{(k)}\}$, \cref{eqn:finite length} in fact shows the finite length of $\{\snorm{\bar \z^{(k+1)}-\z^{(k)}}\}$. Based upon \cref{eqn:finite length of sequence} and \Cref{prop:properties of acc pt set} (i), we complete the proof.
	\end{proof}
	
	\begin{rem}\label{rem:how to choose eps}
		\par Lifting the \L ojasiewicz exponent $\theta$ is crucial in proving the iterate convergence when the unlifted one equals $0$. Just take a look at \cref{eqn:ub for iter gap}. If we keep using $\theta=0$, we would merely obtain the square summability of $\{\snorm{\bar\z^{(k+1)}-\z^{(k)}}$ using telescoping-summation arguments. This point is fairly different from the \palme~because, in the latter case, $u^{(k)}\equiv0$ and \cref{eqn:ub for iter gap} would then give rise to a finite termination at a stationary point \cite{bolte2014proximal2,Xu2013A2}. 
	\end{rem}

	\section{The Asymptotic Convergence Rates of \palmi}\label{sec:convergence rates of ipalm}
	
	\par In this part, we investigate the asymptotic convergence rates of \palmi~on the basis of \Cref{thm:global convergence}; henceforth, $\theta$ refers to the \L ojasiewicz expoenent of $F$ at the unique limit point of $\{\z^{(k)}\}$. To derive specific rates, we consider both \textit{exponentially} and \textit{sublinearly} decreasing $\{\eps^{(k)}\}$. 
	
	\par For notational convenience, let 
	\begin{equation*}
		S^{(t)}:=\sum_{k=t}^\infty\snorm{\bar\z^{(k+1)}-\z^{(k)}}.
	\end{equation*}
	Under the assumptions made in \Cref{thm:global convergence}, $S^{(0)}<\infty$, $\sum_{k=1}^\infty\snorm{\Delta\z^{(k)}}<\infty$, and $\{\z^{(k)}\}$ converges to some stationary point $\bar\z$ of $F$ in \cref{eqn:extended form}. It is then easy to get that, for any $t\ge0$,
	\begin{equation}
		\snorm{\z^{(t)}-\bar\z}\le\sum_{k=t}^\infty\lrbrace{\snorm{\bar\z^{(k+1)}-\z^{(k)}}+\snorm{\bar\z^{(k+1)}-\z^{(k+1)}}}\le S^{(t)}+\omega\sum_{k=t}^\infty\eps^{(k+1)},
		\label{eqn:rate start point 1}
	\end{equation}
	where the second inequality is from \Cref{lem:error bound}. 
	
	\par If there exists $K\in\N$ such that $v^{(K)}=\bbar F$, the asymptotic convergence rates depends only on the choices of $\{\eps^{(k)}\}$ because $S^{(t)}\equiv0$ for all sufficiently large $t$ by \Cref{prop:suff reduce} (iv). We then readily have the following result, whose proof is omitted.
	
	\begin{thm}\label{thm:surrogate finite stop}
		Suppose the assumptions in \Cref{thm:global convergence} hold. Let $\bar\z$ be the unique limit point of the sequence $\{\z^{(k)}\}$ generated by \palmi. Assume that there exists $K\in\N$ such that $v^{(K)}=\bbar F$. 
		\begin{enumerate}[label=(\roman*),topsep=0pt, parsep=0pt, itemsep=0pt]
			\item If $\eps^{(k)}=\bar\eps\tilde\rho^k$ for any $k\ge0$, where $\tilde\rho\in(0,1)$, then $\snorm{\z^{(k)}-\bar\z}\le\mathcal{O}(\tilde\rho^k)$ for any $k\ge K$.
			
			\item If $\eps^{(k)}=\frac{\bar\eps}{(k+1)^\ell}$ for any $k\ge0$, where $\ell>1$, then $\snorm{\z^{(k)}-\bar\z}\le\mathcal{O}\lrbracket{k^{-(\ell-1)}}$ for any $k\ge K$.
		\end{enumerate}
	\end{thm}
	
	\par Due to the errors in solving \cref{eqn:prox-linear subprob new}, $v^{(k)}=\bbar F$ can hardly happen in implementation. In the remainder of this section, we focus on the cases where $v^{(k)}-\bbar F>0$ for any $k\ge0$. We first derive a universal upper bound on $S^{(t)}$ in this setting.
	
	\begin{lem}\label{lem:ub on S}
		Suppose the assumptions in \Cref{thm:global convergence} hold and $\bar\theta\in[\theta,1)$. Assume that $v^{(k)}-\bbar F>0$ for any $k\ge0$. Then we have, for any $t\ge1$,
		\begin{equation*}
			S^{(t)}\le \lrsquare{S^{(t-1)}-S^{(t)}}+C_2\lrsquare{S^{(t-1)}-S^{(t)}+\omega\eps^{(t)}}^{\frac{1-\bar\theta}{\bar\theta}}+C_3E_{\bar\theta}^{(t)},
		\end{equation*}
		where
		\begin{align}
			E_{\bar\theta}^{(t)}&:=\sum_{k=t}^\infty\eps^{(k)}+\sum_{k=t}^\infty(e^{(k)})^{\bar\theta}+(e^{(t)})^{1-\bar\theta},\nonumber\\
			p&:=2\sqrt{\frac{c}{(\gamma-1)L}}\bar M,\quad\:\:\: q:=\frac{\sqrt{c(\gamma-1)L}}{2},\label{eqn:constants C3 C4}\\
			C_2&:=\frac{cp\lrbracket{c\bar M}^{\frac{1-\bar\theta}{\bar\theta}}}{q(1-\bar\theta)},\qquad\: C_3:=\max\lrbrace{\omega,\frac{\bar C_1^{\bar\theta}}{pq},\frac{cp\bar C_1^{1-\bar\theta}}{q(1-\bar\theta)}}.\nonumber
		\end{align}
	\end{lem}
	\begin{proof}
		In view of \cref{eqn:finite length} in the proof of \Cref{thm:global convergence}, one has the following upper bound on $S^{(t)}$:
		\begin{align}
			S^{(t)}\le&\snorm{\bar\z^{(t)}-\z^{(t-1)}}+\omega\sum_{k=t}^\infty\eps^{(k)}+\frac{\bar C_1^{\bar\theta}}{ p q}\sum_{k=t}^\infty(e^{(k)})^{\bar\theta}+\frac{ c p}{ q(1-\bar\theta)}\lrsquare{F(\z^{(t)})-\bbar F+2\sum_{k=t}^\infty C_1^{(t)}\snorm{\Delta\z^{(k)}}^2}^{1-\bar\theta}\nonumber\\
			\le&\lrsquare{S^{(t-1)}-S^{(t)}}+\omega\sum_{k=t}^\infty\eps^{(k)}+\frac{\bar C_1^{\bar\theta}}{ p q}\sum_{k=t}^\infty(e^{(k)})^{\bar\theta}+\frac{ c p}{ q(1-\bar\theta)}\left[\abs{F(\z^{(t)})-\bbar F}^{1-\bar\theta}+\bar C_1^{1-\bar\theta}(e^{(t)})^{1-\bar\theta}\right],\label{eqn:rate start point 2}
		\end{align}
		where the second inequality follows from \cref{eqn:constants}, \Cref{lem:inequality lemma} (iii), \Cref{lem:error bound}, and the summability of $\{\eps^{(k)}\}$. By \Cref{coro:larger theta} and \Cref{prop:approx subgrad lb}, we obtain that
		\begin{align}
			&\abs{F(\z^{(t)})-\bbar F}^{\bar\theta}\le c\cdot\dist(0,\partial F(\z^{(t)}))\le c\snorm{\w^{(t)}}\nonumber\\
			&\le c\bar M\lrsquare{\snorm{\z^{(t)}-\z^{(t-1)}}+\snorm{\Delta\z^{(t)}}}\le c\bar M\lrsquare{S^{(t-1)}-S^{(t)}+\omega\eps^{(t)}},\nonumber
		\end{align}
		where the last inequality uses \Cref{lem:error bound}. Therefore, since $\frac{1-\bar\theta}{\bar\theta}>0$, 
		\begin{equation}
			\abs{F(\z^{(t)})-\bbar F}^{1-\bar\theta}\le\lrbracket{c\bar M}^{\frac{1-\bar\theta}{\bar\theta}}\lrsquare{S^{(t-1)}-S^{(t)}+\omega\eps^{(t)}}^{\frac{1-\bar\theta}{\bar\theta}}.
			\label{eqn:rate start point 3}
		\end{equation}
		Plugging \cref{eqn:rate start point 3} into \cref{eqn:rate start point 2}, we conclude that
		\begin{align}
			S^{(t)}\le&\lrsquare{S^{(t-1)}-S^{(t)}}+\frac{cp\lrbracket{c\bar M}^{\frac{1-\bar\theta}{\bar\theta}}}{ q(1-\bar\theta)}\lrsquare{S^{(t-1)}-S^{(t)}+\omega\eps^{(t)}}^{\frac{1-\bar\theta}{\bar\theta}}\nonumber\\
			&+\omega\sum_{k=t}^\infty\eps^{(k)}+\frac{\bar C_1^{\bar\theta}}{ p q}\sum_{k=t}^\infty(e^{(k)})^{\bar\theta}+\frac{c p \bar C_1^{1-\bar\theta}}{ q(1-\bar\theta)}(e^{(t)})^{1-\bar\theta}.\nonumber
		\end{align}
		This completes the proof by noting the definition of $C_2$, $C_3$, and $E_{\bar\theta}^{(t)}$ in \cref{eqn:constants C3 C4}. 
	\end{proof}
	
	\par Below, given $v^{(k)}-\bbar F>0$ for any $k\ge0$, we present the asymptotic convergence rates of \palmi~under different values of $\theta$ and choices of $\{\eps^{(k)}\}$. Before that, we give a technical lemma, whose proof could be found in Lemma 4 of \cite[Chapter 2]{polyak1987introduction}. 
	
	\begin{lem}\label{lem:chung's inequality}
		Let $\{a^{(k)}\}$ be a nonnegative sequence. If 
		$$a^{(k+1)}\le\lrbracket{1-\frac{b}{k}}a^{(k)}+\frac{d}{k^{s+1}},$$ 
		where $b$, $d$, $s$ are positive scalars, and $b>s$, then 
		$$a^{(k)}\le\frac{d}{b-s}\frac{1}{k^s}+o\lrbracket{\frac{1}{k^s}}.$$
	\end{lem}
	
	\par We begin with exponentially decreasing $\{\eps^{(k)}\}$.
	
	\begin{thm}\label{thm:convergence rates exponential}
		Suppose the assumptions in \Cref{thm:global convergence} hold with $\eps^{(k)}=\bar\eps\tilde\rho^k$ for any $k\ge0$, where $\tilde\rho\in(0,1)$. Let $\bar\z$ be the unique limit point of the sequence $\{\z^{(k)}\}$ generated by \palmi~and $\theta\in[0,1)$ is the \L ojasiewicz exponent of $F$ at $\bar\z$. Assume that $v^{(k)}>\bbar F$ for any $k\ge0$. 
		\begin{enumerate}[label=(\roman*),topsep=0pt, parsep=0pt, itemsep=0pt]
			\item If $\theta=0$, then there exists $\rho_1\in(0,1)$ such that $\snorm{\z^{(k)}-\bar\z}\le\mathcal{O}(\rho_1^k)$ for all sufficiently large $k$.
			\item If $\theta\in(0,\frac{1}{2}]$, then there exists $\rho_2\in(0,1)$ such that $\snorm{\z^{(k)}-\bar\z}\le\mathcal{O}(\rho_2^k)$ for all sufficiently large $k$.
			\item If $\theta\in(\frac{1}{2},1)$, then $\snorm{\z^{(k)}-\bar\z}\le\mathcal{O}\lrbracket{k^{-\frac{1-\theta}{2\theta-1}}}$ for all sufficiently large $k$.
		\end{enumerate}
	\end{thm}
	
	\begin{proof}[Proof of (i) and (ii)]
		\par By the choice of $\{\eps^{(k)}\}$, any $\bar\theta\in(0,\frac{1}{2}]$: $\bar\theta\ge\theta$ complies with \Cref{cond:eps} (b). Hence, \cref{lem:ub on S} is valid for any $\bar\theta\in(0,\frac{1}{2}]$: $\bar\theta\ge\theta$. From the proof of \Cref{prop:properties of acc pt set} (i) and the choice of $\{\eps^{(k)}\}$, there exists $k_2\in\N$: $k_2\ge k_1$ such that $S^{(t-1)}-S^{(t)}+\omega\eps^{(t)}\in[0,1)$ for any $t\ge k_2$. Since $\frac{1-\bar\theta}{\bar\theta}\ge1$ for any $\bar\theta\in(0,\frac{1}{2}]$, one has
		$$\lrsquare{S^{(t-1)}-S^{(t)}+\omega\eps^{(t)}}^{\frac{1-\bar\theta}{\bar\theta}}\le S^{(t-1)}-S^{(t)}+\omega\eps^{(t)},\quad\forall~t\ge k_2.$$
		Combining the last inequality and \Cref{lem:ub on S}, we obtain that 
		\begin{equation*}
			S^{(t)}\le\bar\rho S^{(t-1)}+\bar C_2\eps^{(t)}+\bar C_3E_{\bar\theta}^{(t)},\quad\forall~t\ge k_2,
		\end{equation*}
		where $\bar\rho:=\frac{1+C_2}{2+C_2}\in(0,1)$, $\bar C_2:=\frac{\omega C_2}{2+C_2}$, $\bar C_3:=\frac{C_3}{2+C_2}$. Invoking the above recursion repeatedly, together with the choice of $\{\eps^{(k)}\}$, yields, for any $t\ge k_2$, 
		\begin{align}
			S^{(t)}\le&\bar\rho^{t-k_2+1}S^{(k_2-1)}+\bar C_2\sum_{k=k_2}^t\bar\rho^{t-k}\tilde\rho^k+\bar C_3\sum_{k=k_2}^t\bar\rho^{t-k}E_{\bar\theta}^{(k)}\nonumber\\
			\le&\bar\rho^{t-k_2+1}S^{(k_2-1)}+\bar C_2t\max\{\bar\rho,\tilde\rho\}^t+\bar C_3\sum_{k=k_2}^t\bar\rho^{t-k}E_{\bar\theta}^{(k)}\label{eqn:linear rate 2}.
		\end{align}
		
		\par We then proceed with calculations on $E_{\bar\theta}^{(t)}$: for any $t\ge1$,
		\begin{equation*}
			E_{\bar\theta}^{(t)}=\sum_{k=t}^\infty\eps^{(k)}+\sum_{k=t}^\infty(e^{(k)})^{\bar\theta}+(e^{(t)})^{1-\bar\theta}=\frac{\bar\eps\tilde\rho^t}{1-\tilde\rho}+\frac{\bar\eps^{2\bar\theta}\tilde\rho^{2\bar\theta t}}{(1-\tilde\rho^2)^{\bar\theta}(1-\tilde\rho^{2\bar\theta})}+\frac{\bar\eps^{2(1-\bar\theta)}\tilde\rho^{2(1-\bar\theta)t}}{(1-\tilde\rho^2)^{1-\bar\theta}}.\nonumber
		\end{equation*}
		Since $\bar\theta\in(0,\frac{1}{2}]$, $2\bar\theta\le\min\{2(1-\bar\theta),1\}$. Therefore, there exists a positive constant $\bar M_1$ such that $E_{\bar\theta}^{(t)}\le\bar M_1\tilde\rho^{2\bar\theta t}$ for any $t\ge1$. Putting this into \cref{eqn:linear rate 2}, we achieve 
		\begin{equation*}
			S^{(t)}\le\bar\rho^{t-k_2+1}S^{(k_2-1)}+\bar C_2t\max\{\bar\rho,\tilde\rho\}^t+\bar C_3\bar M_1t\max\{\bar\rho,\tilde\rho^{2\bar\theta}\}^t,\quad\forall~t\ge k_2.
		\end{equation*}
		Since $\max\{\bar\rho,\tilde\rho^{2\bar\theta}\}\in(0,1)$, there exists an integer $k_3\in\N$: $k_3\ge k_2$ such that $S^{(t)}\le\mathcal{O}\lrbracket{\max\{\bar\rho,\tilde\rho^{2\bar\theta}\}^{\frac{t}{2}}}$ for all $t\ge k_3$. In view of \cref{eqn:rate start point 1}, we prove statements (i) and (ii).
	\end{proof}
	
	\begin{proof}[Proof of (iii)]
		\par By the choice of $\{\eps^{(k)}\}$, $\bar\theta=\theta$ just complies with \Cref{cond:eps} (b) and hence \Cref{lem:ub on S} is valid. From the proof of \Cref{prop:properties of acc pt set} (i) and the choice of $\{\eps^{(k)}\}$, there exists $k_4\in\N$: $k_4\ge k_1$ such that $S^{(t-1)}-S^{(t)}+\omega\eps^{(t)}\in[0,1)$ for any $t\ge k_4$. Since $\theta\in(\frac{1}{2},1)$, $\frac{1-\theta}{\theta}<1$. Therefore, invoking \Cref{lem:inequality lemma} (iii),
		$$\lrsquare{S^{(t-1)}-S^{(t)}+\omega\eps^{(t)}}^{\frac{1-\theta}{\theta}}\le\lrsquare{S^{(t-1)}-S^{(t)}}^{\frac{1-\theta}{\theta}}+\omega(\eps^{(t)})^{\frac{1-\theta}{\theta}}.$$
		Moreover, for any $t\ge k_4$,  
		$$S^{(t-1)}-S^{(t)}\le\lrsquare{S^{(t-1)}-S^{(t)}}^{\frac{1-\theta}{\theta}}.$$
		Combining the above two relations with \Cref{lem:ub on S}, one has, for any $t\ge k_4$,
		\begin{equation}
			S^{(t-1)}=\lrbracket{S^{(t-1)}-S^{(t)}}+S^{(t)}\le(2+C_2)\lrsquare{S^{(t-1)}-S^{(t)}}^{\frac{1-\theta}{\theta}}+\omega C_2(\eps^{(t)})^{\frac{1-\theta}{\theta}}+C_3E_{\theta}^{(t)}.
			\label{eqn:sublinear rate 3}
		\end{equation}
		With simple calculations, we reach
		\begin{equation*}
			\omega C_2(\eps^{(t)})^{\frac{1-\theta}{\theta}}+C_3E_{\theta}^{(t)}=\omega C_2(\bar\eps\tilde\rho^t)^{\frac{1-\theta}{\theta}}+C_3\lrsquare{\frac{\bar\eps\tilde\rho^t}{1-\tilde\rho}+\frac{\bar\eps^{2\theta}\tilde\rho^{2\theta t}}{(1-\tilde\rho^{2\theta})(1-\tilde\rho^2)^\theta}+\frac{\bar\eps^{2(1-\theta)}\tilde\rho^{2(1-\theta)t}}{(1-\tilde\rho^2)^\theta}}.
		\end{equation*}
		Putting the last equality into \cref{eqn:sublinear rate 3}, we have the existence of some $\bar M_2>0$ for which
		\begin{equation*}
			S^{(t-1)}\le(2+C_2)\lrsquare{S^{(t-1)}-S^{(t)}}^{\frac{1-\theta}{\theta}}+\bar M_2\tilde\rho^{\frac{1-\theta}{\theta}t},\quad\forall~t\ge k_4.
		\end{equation*}
		Since $\frac{\theta}{1-\theta}>1$, invoking \Cref{lem:inequality lemma} (ii), we obtain from the last inequality that
		$$(S^{(t-1)})^{\frac{\theta}{1-\theta}}\le2^{\frac{2\theta-1}{1-\theta}}\lrsquare{(2+C_2)^{\frac{\theta}{1-\theta}}\big(S^{(t-1)}-S^{(t)}\big)+\bar M_2^{\frac{\theta}{1-\theta}}\tilde\rho^t},\quad\forall~t\ge k_4,$$
		which further yields
		\begin{equation}
			S^{(t)}\le S^{(t-1)}-C_4(S^{(t-1)})^{\frac{\theta}{1-\theta}}+C_5\tilde\rho^t,\quad\forall~t\ge k_4,
			\label{eqn:sublinear rate 5}
		\end{equation}
		where $C_4:=2^{-\frac{2\theta-1}{1-\theta}}(2+C_2)^{-\frac{\theta}{1-\theta}}$, $C_5:=\bar M_2^{\frac{\theta}{1-\theta}}(2+C_2)^{-\frac{\theta}{1-\theta}}$. 
		
		\par Let $h_{\theta}:\R_+\to\R$ be defined as $h_{\theta}(x):=x^{\frac{\theta}{1-\theta}}$. Since $\frac{\theta}{1-\theta}>1$, $h_{\theta}$ is convex on $\R_+$. Hence, for a fixed $\xi\ge0$,
		\begin{align}
			&(S^{(t-1)})^{\frac{\theta}{1-\theta}}-(\xi t^{-\frac{1-\theta}{2\theta-1}})^{\frac{\theta}{1-\theta}}=h_{\theta}(S^{(t-1)})-h_{\theta}(\xi t^{-\frac{1-\theta}{2\theta-1}})\nonumber\\
			\ge&h_{\theta}'(\xi t^{-\frac{1-\theta}{2\theta-1}})\lrsquare{S^{(t-1)}-\xi t^{-\frac{1-\theta}{2\theta-1}}}=\frac{\theta\xi^{\frac{2\theta-1}{1-\theta}}}{(1-\theta)t}\lrsquare{S^{(t-1)}-\xi t^{-\frac{1-\theta}{2\theta-1}}}.\nonumber
		\end{align}
		Plugging the last inequality into \cref{eqn:sublinear rate 5}, one has, for any $t\ge k_4$ (possibly after enlargement),
		\begin{align}
			S^{(t)}&\le S^{(t-1)}-C_4\left[(S^{(t-1)})^{\frac{\theta}{1-\theta}}-(\xi t^{-\frac{1-\theta}{2\theta-1}})^{\frac{\theta}{1-\theta}}\right]-C_4(\xi t^{-\frac{1-\theta}{2\theta-1}})^{\frac{\theta}{1-\theta}}+C_5\tilde\rho^t\nonumber\\
			&\le S^{(t-1)}-\frac{C_4\theta\xi^{\frac{2\theta-1}{1-\theta}}}{(1-\theta)t}\lrsquare{S^{(t-1)}-\xi t^{-\frac{1-\theta}{2\theta-1}}}-C_4(\xi t^{-\frac{1-\theta}{2\theta-1}})^{\frac{\theta}{1-\theta}}+C_5\tilde\rho^t\nonumber\\
			&=\left[1-\frac{C_4\theta\xi^{\frac{2\theta-1}{1-\theta}}}{(1-\theta)t}\right]S^{(t-1)}+\frac{C_4\frac{2\theta-1}{1-\theta}\xi^{\frac{\theta}{1-\theta}}}{t^{\frac{\theta}{2\theta-1}}}+C_5\tilde\rho^t\nonumber\\
			&\le\left[1-\frac{C_4\theta\xi^{\frac{2\theta-1}{1-\theta}}}{(1-\theta)t}\right]S^{(t-1)}+\frac{C_4\frac{2\theta-1}{1-\theta}\xi^{\frac{\theta}{1-\theta}}+C_5}{t^{\frac{\theta}{2\theta-1}}}.\nonumber
		\end{align}
		Thus, after choosing $\xi$ such that $C_4\theta\xi^{\frac{2\theta-1}{1-\theta}}>\frac{1-\theta}{2\theta-1}$, we conclude from \Cref{lem:chung's inequality} that
		$$S^{(t)}\le\mathcal{O}\lrbracket{t^{-\frac{1-\theta}{2\theta-1}}},\quad\forall~t\ge k_4,$$
		which completes the proof of statement (iii) after combination with \cref{eqn:rate start point 1}.
	\end{proof}
	
	\par Due to the solution errors in solving \cref{eqn:prox-linear subprob new}, the finite termination of \palme~when $\theta=0$ seems to go beyond the reach of \palmi, no matter how fast $\{\eps^{(k)}\}$ decreases. Note that we separate statements (i) and (ii) in \Cref{thm:convergence rates exponential} to indicate that $\rho_1$ and $\rho_2$ can take different values.
	
	\par When $\{\eps^{(k)}\}$ decays sublinearly, only sublinear rates are achievable, regardless of the value of $\theta$. Below, we first give an auxiliary lemma, whose proof is straightforward and thus omitted.
	
	\begin{lem}\label{lem:tau}
		Let $\theta\in(\frac{1}{2},1)$ and $\ell>\frac{\theta+1}{2\theta}$. Define
		$$\tau(\theta,\ell):=\min\lrbrace{\frac{1-\theta}{\theta}\ell,\ell-1,(2\ell-1)\theta-1,(2\ell-1)(1-\theta)}.$$
		Then $\tau(\theta,\ell)$ has the following closed-form expression:
		\begin{equation*}
			\tau(\theta,\ell)=\left\{\begin{array}{ll}
				\frac{1-\theta}{\theta}\ell, & \text{if }\ell\in\left[\frac{\theta}{2\theta-1},\infty\right),\\
				(2\ell-1)\theta-1, & \text{if }\ell\in\left(\frac{\theta+1}{2\theta},\frac{\theta}{2\theta-1}\right].
			\end{array}\right.
		\end{equation*}
	\end{lem}
	
	\begin{thm}\label{thm:convergence rates sublinear}
		Suppose the assumptions in \Cref{thm:global convergence} hold with $\eps^{(k)}=\frac{\bar\eps}{(k+1)^{\ell}}$ for any $k\ge0$, where $\ell>1$. Let $\bar\z$ be the unique limit point of the sequence $\{\z^{(k)}\}$ generated by \palmi~and $\theta\in[0,1)$ is the \L ojasiewicz exponent of $F$ at $\bar\z$. Assume that $v^{(k)}>\bbar F$ for any $k\ge0$. Then, for all sufficiently large $k$,
		\begin{equation}
			\snorm{\z^{(k)}-\bar\z}\le\left\{\begin{array}{ll}
				\mathcal{O}\lrbracket{k^{-\frac{1-\theta}{2\theta-1}}}, & \text{if }\ell\in\left[\frac{\theta}{2\theta-1},\infty\right)\text{ and }\theta\in(\frac{1}{2},1),\\
				\mathcal{O}\lrbracket{k^{-(\ell-1)}}, & \text{otherwise}.
			\end{array}\right.
			\label{eqn:fix ell best rate}
		\end{equation}
	\end{thm}
	
	\begin{proof}
		\par The proof of this theorem is analogous to that of statement (iii) in \Cref{thm:convergence rates exponential}. By \Cref{rem:choice of eps}, any $\bar\theta\in[\theta,1)$: $\bar\theta>\max\{\sqrt{\frac{1}{2\ell-1}},\frac{1}{2}\}$ complies with \Cref{cond:eps} (b). For such $\bar\theta$, \Cref{lem:ub on S} is valid. From the proof of \cref{prop:properties of acc pt set} (i) and the choice of $\{\eps^{(k)}\}$, there exists $k_5\in\N$ such that $S^{(t-1)}-S^{(t)}+\omega\eps^{(t)}\in[0,1)$ for any $t\ge k_5$. By the choice of $\{\eps^{(k)}\}$,
		\begin{align*}
			&\omega C_2(\eps^{(t)})^{\frac{1-\bar\theta}{\bar\theta}}+C_3E_{\bar\theta}^{(t)}\\
			=&\omega C_2(\eps^{(t)})^{\frac{1-\bar\theta}{\bar\theta}}+C_3\lrsquare{\sum_{k=t}^\infty\eps^{(k)}+\sum_{k=t}^\infty(e^{(k)})^{\bar\theta}+(e^{(t)})^{1-\bar\theta}}\\
			\le&\frac{\omega C_2}{(t+1)^{\frac{1-\bar\theta}{\bar\theta}\ell}}+C_3\lrsquare{\sum_{k=t}^\infty\frac{\bar\eps}{(k+1)^\ell}+\sum_{k=t}^\infty\frac{\bar\eps^{2\bar\theta}}{(2\ell-1)^{\bar\theta}k^{(2\ell-1)\bar\theta}}+\lrbracket{\sum_{k=t}^\infty\frac{\bar\eps^2}{(k+1)^{2\ell}}}^{1-\bar\theta}}\\
			\le&\frac{\omega C_2}{t^{\frac{1-\bar\theta}{\bar\theta}\ell}}+C_3\lrsquare{\frac{\bar\eps}{(\ell-1)t^{\ell-1}}+\frac{\frac{\bar\eps^{2\bar\theta}}{(2\ell-1)^{\bar\theta}[(2\ell-1)\bar\theta-1]}}{(t-1)^{(2\ell-1)\bar\theta-1}}+\frac{\bar\eps^{2(1-\bar\theta)}}{(2\ell-1)^{1-\bar\theta}t^{(2\ell-1)(1-\bar\theta)}}}.
		\end{align*}
		Noticing the definition of $\tau$ in \Cref{lem:tau} and putting the last inequality into \cref{eqn:sublinear rate 3}, we have the existence of some $\bar M_3>0$ for which
		\begin{equation*}
			S^{(t-1)}\le(2+C_2)\lrsquare{S^{(t-1)}-S^{(t)}}^{\frac{1-\bar\theta}{\bar\theta}}+\frac{\bar M_3}{t^{\tau(\bar\theta,\ell)}},\quad\forall~t\ge k_5.
		\end{equation*}
		Following the similar arguments in the proof of statement (iii) in \Cref{thm:convergence rates exponential}, one could obtain
		$$S^{(t)}\le\left[1-\frac{C_4\bar\theta\xi^{\frac{2\bar\theta-1}{1-\bar\theta}}}{(1-\bar\theta)t}\right]S^{(t-1)}+\frac{C_4\frac{2\bar\theta-1}{1-\bar\theta}\xi^{\frac{\bar\theta}{1-\bar\theta}}+C_6}{t^{\min\{\frac{\bar\theta}{2\bar\theta-1},\tau(\bar\theta,\ell)\frac{\bar\theta}{1-\bar\theta}\}}},\quad\forall~t\ge k_5,$$
		where $C_6:=\bar M_3^{\frac{\bar\theta}{1-\bar\theta}}(2+C_2)^{-\frac{\theta}{1-\theta}}$. Note that $\min\{\frac{\bar\theta}{2\bar\theta-1},\tau(\bar\theta,\ell)\frac{\bar\theta}{1-\bar\theta}\}>1$ due to $\bar\theta>\sqrt{\frac{1}{2\ell-1}}$. Thus, after choosing $\xi$ such that $C_4\bar\theta\xi^{\frac{2\bar\theta-1}{1-\bar\theta}}>\min\{\frac{\bar\theta}{2\bar\theta-1},\tau(\bar\theta,\ell)\frac{\bar\theta}{1-\bar\theta}\}-1$, we conclude from \Cref{lem:chung's inequality} that
		$$S^{(t)}\le\mathcal{O}\lrbracket{t^{-[\min\{\frac{\bar\theta}{2\bar\theta-1},\tau(\bar\theta,\ell)\frac{\bar\theta}{1-\bar\theta}\}-1]}},\quad\forall~t\ge k_5.$$
		Invoking \Cref{lem:tau} and combining \cref{eqn:rate start point 1}, it is not difficult to derive
		\begin{equation}
			\snorm{\z^{(t)}-\bar\z}\le\left\{\begin{array}{ll}
				\mathcal{O}\lrbracket{t^{-\frac{1-\bar\theta}{2\bar\theta-1}}}, & \text{if }\ell\in[\frac{\bar\theta}{2\bar\theta-1},\infty),\\
				\mathcal{O}\lrbracket{t^{-\frac{2\bar\theta^2\ell-\bar\theta^2-1}{1-\bar\theta}}}, & \text{if }\ell\in(\frac{\bar\theta^2+1}{2\bar\theta^2},\frac{\bar\theta}{2\bar\theta-1}),
			\end{array}\right.\quad\forall t\ge k_5.
			\label{eqn:rate for any theta}
		\end{equation}
		
		\par Since \cref{eqn:rate for any theta} holds for any $\bar\theta\in[\theta,1):\bar\theta>\max\{\sqrt{\frac{1}{2\ell-1}},\frac{1}{2}\}$ and $k_5$ does not rely on its value, the best rate exponent must be attained at one of the following two:
		$$\text{(A)}~\max_{\bar\theta\in[\frac{\ell}{2\ell-1},1),\bar\theta\ge\theta}~\frac{1-\bar\theta}{2\bar\theta-1};\quad\text{(B)}~\max_{\bar\theta\in(\sqrt{\frac{1}{2\ell-1}},\frac{\ell}{2\ell-1}],\bar\theta\ge\theta}~\frac{2\bar\theta^2\ell-\bar\theta^2-1}{1-\bar\theta}.$$ 
		Note that $\bar\theta\ge\frac{\ell}{2\ell-1}$ holds for any $\bar\theta\ge\theta$ if and only if $\ell\ge\frac{\theta}{2\theta-1}$ and $\theta\in(\frac{1}{2},1)$. Suppose $\ell\ge\frac{\theta}{2\theta-1}$ and $\theta\in(\frac{1}{2},1)$. Then the best rate exponent is achieved at (A) with just $\bar\theta=\theta$, establishing the first line in \cref{eqn:fix ell best rate}. Suppose otherwise, it is easy to check that the optimal values of both (A) and (B) are $\ell-1$ with the minimizer $\bar\theta=\frac{\ell}{2\ell-1}$, leading to the second line of \cref{eqn:fix ell best rate}. The proof is complete.
	\end{proof}
	
	\par The first line of \cref{eqn:fix ell best rate} recovers the result for \palme~\cite{bolte2014proximal2,Xu2013A2}. In view of this and statement (iii) in \Cref{thm:convergence rates exponential}, it appears that the solutions errors in solving \cref{eqn:prox-linear subprob new} do not affect the asymptotic rates of \palmi~at all if $\{\eps^{(k)}\}$ decreases fast enough. 
	
	\section{Numerical Experiments}\label{sec:numerical experiments}
	
	\par In this section, we use numerical results to validate the convergence of \palmi~and demonstrate its merits over \palme~and \palmf. All the numerical experiments presented are run in a platform with Intel(R) Xeon(R) Gold 6242R CPU @ 3.10GHz and 510GB RAM running \textsc{Matlab} R2018b under Ubuntu 20.04.
	
	\vskip 0.1cm
	
	\subsection{Optimization with Linear Constraints}
	
	\par The first class of problems under consideration is the one discussed in \cite{hu2021global}, i.e., the $\ell_1$ penalized discretized multi-marginal optimal transport problems in $\R^d$ arising from quantum physics, which generally take the form
	\begin{equation}
		\begin{array}{cl}
			\min\limits_{\{X_i\}_{i=2}^N} & \sum\limits_{i=2}^N\inner{X_i,\Lambda C}+\sum\limits_{i<j}\lrbracket{\inner{X_i,\Lambda X_jC}+\beta\inner{X_i,X_j}}\\
			\st & X_i\in\calS:=\{W\in\R^{K\times K}:W\one=\one,~W^\T\varrho=\varrho,~\trace(W)=0,~W\ge0\},~~\forall~i.
		\end{array}
		\label{eqn:app1 quantum physics}
	\end{equation}
	Here, $\beta>0$ is the penalty parameter and $N$, $K\in\N$ refer, respectively, to the number of electrons in the system and finite elements $\mathcal{T}:=\{e_k\}_{k=1}^K\subseteq\R^d$ discretizing a bounded domain $\Omega$. The vector $\varrho:=[\varrho_1,\ldots,\varrho_K]^\T\in\R^K$ is defined as $\varrho_k:=\int_{e_k}\rho(\mb{r})\,\mathrm{d}\mb{r}$ for any $k\in\{1,\ldots,K\}$, where $p:\R^d\to\R_+$ is the single-electron density of the system. The diagonal matrix $\Lambda:=\Diag(\varrho)$ and $C=(C_{ij})$ denotes the discretized Coulomb cost matrix whose diagonal elements are all set to zero to avoid numerical instability:
	$$C_{ij}:=\left\{\begin{array}{ll}
		\snorm{\mb{d}_i-\mb{d}_j}^{-1}, & \text{if }i\ne j,\\
		0, & \text{otherwise}
	\end{array}\right.$$
	with $\{\mb{d}_k\}_{k=1}^K\subseteq\R^d$ being the barycenters of elements $\{e_k\}_{k=1}^K$. For brevity, let $\mathcal{B}:\R^{K\times K}\to\R^{2K+1}$ be a linear operator defined as
	$$\mathcal{B}(W):=[\one^\T W^\T,~\varrho^\T W,~\trace(W)]^\T\in\R^{2K+1},\quad\forall~W\in\R^{K\times K},$$
	and $\mb{b}:=[\one^\T,~\varrho^\T,~0]^\T\in\R^{2K+1}$. Then the set $\calS$ can be expressed simply as $\calS=\{W:\mathcal{B}(W)=\mb{b},~W\ge0\}$. This type of constraints has been mentioned in \Cref{exm:linear constraint}. In our experiments, we consider a 1D system with $N=3$ electrons and the domain $\Omega=[-1,1]$; the density is a normalized Gaussian, namely,
	$$\rho(x)\propto\exp\lrbracket{-x^2/\sqrt{\pi}},\quad\forall~x\in\R.$$
	We adopt an equal-mass discretization so that all the entries in $\varrho$ are identical. 
	
	\par We compare the performances of \palme~and \palmi~when solving \cref{eqn:app1 quantum physics} with $K=36$; that is to say, the number of variables equals $36^2\times2=2592$. The proximal parameter is fixed at $\sigma_i^{(k)}\equiv\sigma=10^{-2}$ for $i=1,\ldots,n$ and $k\ge0$. In both \palme~and \palmi, we adapt the semismooth Newton-CG (\ssncg) proposed in \cite{li2020efficient} to efficiently compute the projection $\PP_{\calS}$. As noted in \Cref{exm:linear constraint}, the infeasibility is inevitable. In this context, particularly with \ssncg~as the subsolver, the residual function $r_i(X_i^{(k+1)},\blam_i^{(k+1)},\tilde X_i^{(k)})$ becomes
	\begin{equation*}
		\max\lrbrace{-\inner{\blam_i^{(k+1)},\mathcal{B}(X_i^{(k+1)})-\mb{b}},0}+\snorm{\mathcal{B}(X_i^{(k+1)})-\mb{b}}_\infty,
		\label{eqn:app1 monitor}
	\end{equation*}
	where $\blam_i^{(k+1)}\in\R^{2K+1}$ is an approximate dual solution given by \ssncg. In \palme, we set $\{\eps^{(k)}\equiv10^{-7}\}_{k\ge0}$ such that all the subproblems are solved to high accuracy, whereas in \palmi, we pick a nonincreasing sequence $\big\{\eps^{(k)}=\max\{\frac{10^{-1}}{(k+1)^{\ell}},10^{-7}\}\big\}_{k\ge0}$ with $\ell=0.75$. Note that by \Cref{assume:compactness}, the stationarity point of \cref{eqn:app1 quantum physics} can be characterized by the Karush-Kuhn-Tucker (KKT) conditions. The outer \palme~or \palmi~framework is therefore stopped once the relative KKT violation is smaller than $10^{-6}$. 
	
	\par We first compare the performances of \palme~and \palmi~with random initializations. The built-in ``\texttt{rand}'' function in \textsc{Matlab} is invoked to generate 100 initial points and then we plot out the averaged history of the relative KKT violation for both \palme~and \palmi; see \Cref{fig:evolve birkhoff} (left). 
	\begin{figure}[htbp]
		\centering
		\includegraphics[width=.75\textwidth]{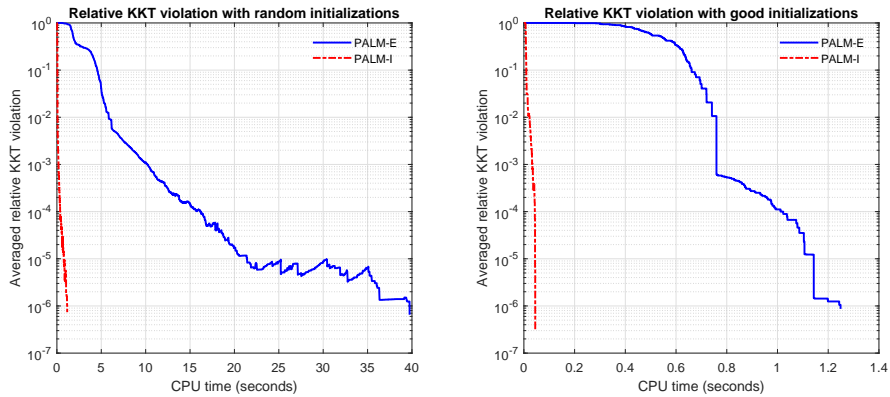}
		\caption{The average relative KKT violation history of \palme~and \palmi~with different types of initializations on \cref{eqn:app1 quantum physics}. Left: random initializations. Right: good initializations.}
		\label{fig:evolve birkhoff}
	\end{figure}
	The average CPU time used by \palme~is about 15.97 seconds, while that of \palmi~is approximately 0.46 seconds. One could then easily conclude the superiority of \palmi~with random initializations in terms of CPU time. Since \cref{eqn:app1 quantum physics} is nonconvex, it is interesting and necessary to inspect the differences between the terminating objective values of \palme~and \palmi. We plot in \Cref{fig:diff birkhoff} with bullets the absolute differences between the terminating objective values and the optimal one $f(Z^*)$. 
	\begin{figure}[htbp]
		\centering
		\includegraphics[width=.7\textwidth]{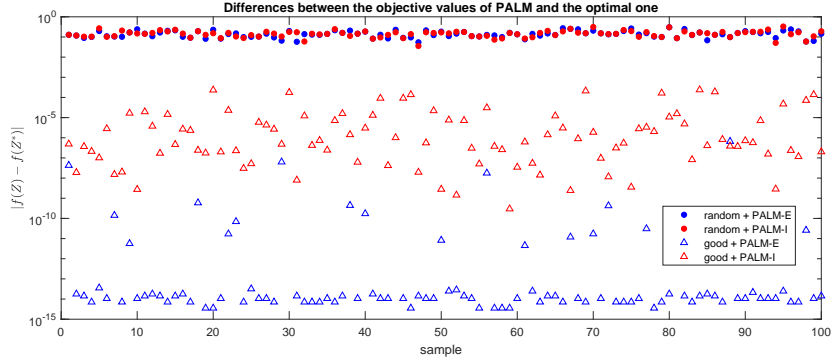}
		\caption{The absolute differences between the objective values given by \palm~and the optimal one. The bullets stand for results with random initializations, while the triangles for results with good initializations.}
		\label{fig:diff birkhoff}
	\end{figure}
	We can observe that, starting with randomly generated initial points, \palme~and \palmi~often stop at points of similar qualities.
	
	\par We then conduct a performance comparison between \palme~and \palmi~with good initializations. The good initial points could be generated by random perturbation around the discretized optimal solution $Z^*$ supplied in \cite{colombo2015multimarginal}. The built-in ``\texttt{rand}'' function is invoked again to generate 100 good initial points, whose quantity of deviation from $Z^*$ is at most $10^{-3}$. We plot out the averaged history of the minimum achieved relative KKT violation for both \palme~and \palmi; see \Cref{fig:evolve birkhoff} (right). Moreover, the average CPU time used by \palme~is about 0.85 seconds, while that of \palmi~is only approximately 0.03 seconds. These reflect the considerable time advantage of \palmi~over \palme~in a neighborhood of optimal solution. Incidentally, the infeasible nature of \palmi~does not ruin much the solution quality; the maximum absolute difference between the terminating objective value and the optimal one is merely $2.39\times10^{-4}$, and it is less than $10^{-5}$ on over 75\% samples.
	
	\par The numerical results in this subsection reflect that \palmi~converges well and is clearly more efficient than \palme~even with infeasibility. The efficiency is brought by the infeasible subsolver \ssncg, whose usage is ensured by our theoretical results. 
	
	\subsection{Optimization with Nonlinear Constraints}
	
	\par\noindent The second class of testing problems involves nonconvex quadratic objective functions and multiple ellipsoidal constraints, having the form
	\begin{equation}
		\begin{array}{cl}
			\displaystyle\min_{\z} & \displaystyle\frac{1}{2}\inner{\z,A\z}+\inner{\mb{b},\z}\\
			\st & \displaystyle\frac{1}{2}\inner{\x_i,B_i\x_i}+\inner{\mb{c}_i,\x_i}\le1,~i=1,\ldots,n
		\end{array}
		\label{eqn:app2 ellipsoid}
	\end{equation}
	where, for $i=1,\ldots,n$, $m_i=m\in\N$; $A\in\mathbb{S}^{mn}$, while $\{B_i\}_{i=1}^n\subseteq\mathbb{S}_{++}^{m}$; $\mb{b}\in\R^{mn}$ and $\{\mb{c}_i\}_{i=1}^n\subseteq\R^{m}$. This problem class is related to several domains \cite{he2010approximation,kuvcera2008convergence,liu2020topology}, as noted in \Cref{exm:nonlinear constraint}. In our implementation, $A$ and $\mb{b}$ are generated by the built-in function ``\texttt{randn}'' in \textsc{Matlab}. To form $\{B_i=(b_{i,jk})\}_{i=1}^n$, we adopt the construction in \cite{dai2006fast,jia2017comparison}: 
	$$b_{i,jk}=10^{\frac{j-1}{m-1}\text{ncond}_i},~\text{if }j=k;~~0,~\text{otherwise}.$$
	It is easy to see that $\text{ncond}_i\in\R_{++}$ controls the condition number and the spectrum of each $B_i$ is spread in $[1,10^{\text{ncond}_i}]$.
	
	\par We compare the performances of \palme, \palmf, and \palmi~when solving \cref{eqn:app2 ellipsoid} with $n=5$ and $m=500$; that is, the number of variables is $2500$. We select $\{\text{ncond}_i\}_{i=1}^n=\{3.00,3.25,3.50,3.75,4.00\}$. The vectors $\{\mb{c}_i\}_{i=1}^n$ are set to be all-zero so that all ellipsoids are concentric. The proximal parameter is fixed at $\sigma_i^{(k)}\equiv\sigma=1$ for $i=1,\ldots,n$ and $k\ge0$. The three algorithms are armed with different subsolvers. Specifically, both \palme~and \palmi~invoke the self-adative alternating direction methods of multiplier proposed in \cite{jia2017comparison} (\sadmm); \palmf~uses the feasible hybrid projection algorithm in \cite{dai2006fast} (\hp). Note that the iterates produced by \sadmm~are not necessarily feasible. For $i=1,\ldots,n$ and $k\ge0$, we terminate \hp~within \palmf~if 
	$$\snorm{\x_i^{(k+1)}-\tilde\x_i^{(k)}+\lambda_{i,\hp}^{(k+1)}(B_i\x_i^{(k+1)}+\mb{c}_i)}\le\frac{\eta}{2}\snorm{\x_i^{(k+1)}-\x_i^{(k)}}_\infty.$$
	where $\eta=0.99\sigma$, and $\lambda_{i,\hp}^{(k+1)}\ge0$ estimates the multiplier associated with the ellipsoidal constraint. It is not difficult to verify that the above inexact criteria help produce iterates fulfilling the assumption in \cite{hua2016block}. Regarding \sadmm~in \palme~and \palmi, the residual function $r_i(\x_i^{(k+1)},\lambda_{i,\sadmm}^{(k+1)},\tilde\x_i^{(k)})$ becomes
	\begin{align*}
		&\max\lrbrace{\inner{\x_i^{(k+1)},\x_i^{(k+1)}-\tilde\x_i^{(k)}+\lambda_{i,\sadmm}^{(k+1)}(B_i\x_i^{(k+1)}+\mb{c}_i)},0}\nonumber\\
		+&\norm{\x_i^{(k+1)}-\tilde\x_i^{(k)}+\lambda_{i,\sadmm}^{(k+1)}(B_i\x_i^{(k+1)}+\mb{c}_i)}_\infty\nonumber\\
		+&\lambda_{i,\sadmm}^{(k+1)}\max\lrbrace{-\lrsquare{\frac{1}{2}\inner{\x_i^{(k+1)},B_i\x_i^{(k+1)}}+\inner{\mb{c}_i,\x_i^{(k+1)}}-\alpha_i},0}\label{eqn:app2 monitor}\\
		+&\max\lrbrace{\frac{1}{2}\inner{\x_i^{(k+1)},B_i\x_i^{(k+1)}}+\inner{\mb{c}_i,\x_i^{(k+1)}}-\alpha_i,0},\nonumber
	\end{align*}
	where $\lambda_{i,\sadmm}^{(k+1)}\in\R_+$ is an approximate dual solution given by \sadmm. In \palme, we set $\{\eps^{(k)}\equiv10^{-6}\}_{k\ge0}$; for \palmi, we choose $\big\{\eps^{(k)}=\max\{\frac{10^{-1}}{(k+1)^{\ell}},10^{-6}\}\big\}_{k\ge0}$ with $\ell=0.75$. As in the previous subsection, the three outer frameworks are stopped once the relative KKT violation is smaller than $10^{-5}$.
	
	\par We invoke the built-in ``\texttt{randn}'' function in \textsc{Matlab} to generate 100 random initial points and then draw the averaged history of the relative KKT violation for the three algorithms; see \Cref{fig:evolve ellipsoid} (left).
	\begin{figure}[htbp]
		\centering
		\includegraphics[width=.8\textwidth]{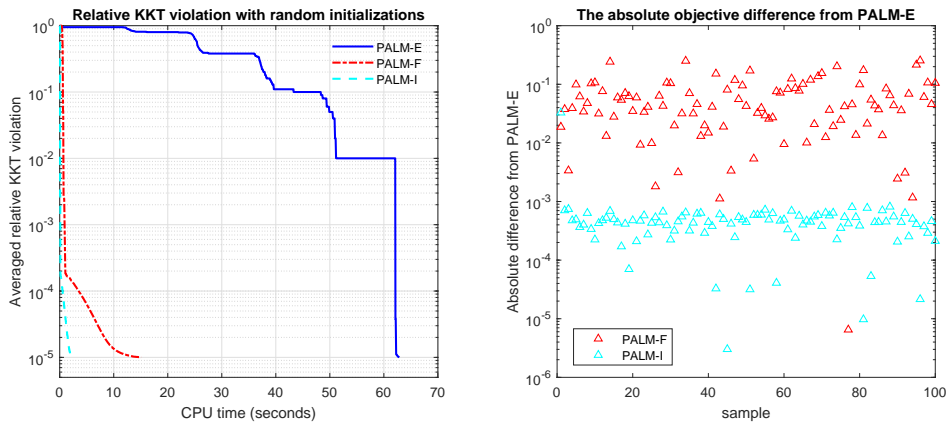}
		\caption{Left: the average relative KKT violation history of \palme, \palmf, and \palmi~with random initializations on \cref{eqn:app2 ellipsoid}. Right: the absolute objective differences from \palme.}
		\label{fig:evolve ellipsoid}
	\end{figure}
	The respective average CPU times used by \palme, \palmf, and \palmi~are approximately 30.15 seconds, 9.26 seconds, and 1.74 seconds. One can observe that \palmi~takes the strengths of the infeasible subsolver \sadmm~and stands out with the best performance. We also make a comparison among the terminating objective values of the three algorithms. Since for \cref{eqn:app2 ellipsoid}, the optimal values are inaccessible, we take those given by \palme~as benchmark and inspect the absolute differences of \palmf~and \palmi~from \palme; see \Cref{fig:evolve ellipsoid} (right). It appears that, even equipped with an infeasible solver, \palmi~is capable of yielding objective values much closer than \palmf~to those of \palme. 
	
	\section{Conclusions}\label{sec:conclusions}
	
	\par We recognize by examples the indispensability of infeasible subsolvers in \palm~whenever constraints are complicated and illustrate through numerical simulations that \palmi~can be far more efficient than \palme~and \palmf. The shortage of existing works on \palmi~motivates us to analyze its convergence properties, particularly in the presence of objective value nonmonotonicity. We achieve this by constructing a monotonically decreasing surrogate sequence. Moreover, an implementable inexact criterion for subsolvers is devised for practical usage.
	
	\par Futural improvements can be anticipated in several lines. For example, one could incorporate nonsmooth regularization terms into objective function and handle infeasibility and nonsmoothness simultaneously. Besides, the assumptions, such as Hoffman-like error bound, may appear to be restrictive and call for further relaxation. Last but not least, it is worth investigating the convergence properties of \palmi~on problems with nonconvex constraints and designing implementable inexact criteria for those contexts.

	\normalem
	\bibliographystyle{siam}
	\bibliography{ref}
	
\end{document}